%% file: HausTipSubmit.tex
\documentclass[11pt]{article}
\usepackage{amsfonts,amsmath,amsthm,amssymb}
\usepackage{latexsym}
\usepackage{graphicx}
\usepackage{color}
\usepackage{wasysym}
\usepackage[utf8]{inputenc}
\usepackage{xspace}

\author{Neil Dobbs\\
\small{School of Mathematics and Statistics}\\
\small{University College, Dublin}\\
\small{ Ireland}\\
\and
Jacek Graczyk\\
\small{Univ. de Paris-Sud}\\
\small{Lab. de Math\'ematiques}\\
\small{91405 Orsay, France}\\
\and
Nicolae Mihalache\\
\small{ Universit\'e Paris-Est Creteil, UMR 8050 CNRS, LAMA, F-94010 Creteil, France }\\
\small{ Universit\'e Gustave Eiffel, LAMA, F-77447 Marne-la-Vall\'ee, France}}

\newtheorem{lem}{Lemma}[section]

\newtheorem{remark}[lem]{Remark}
\newtheorem{theo}{Theorem}
\newtheorem{coro}[lem]{Corollary}
\newtheorem{prop}[lem]{Proposition}

\newtheorem{defi}[lem]{Definition}

\newcommand{\M}{{\cal M}}

\newcommand{\calD}{{\mathcal D}}
\newcommand{\cD}{{\mathcal D}}

\font\mathfonta=msam10 at 11pt
\font\mathfontb=msbm10 at 11pt
\def\Bbb#1{\mbox{\mathfontb #1}}
\def\lesssim{\mbox{\mathfonta.}}

\def\grtsim{\mbox{\mathfonta\&}}

\def\diam{{\mathrm{diam}\,}}
\def\Leb{{\mathrm{Leb}}}
\def\dist#1{{{\mathrm{dist}}\br{#1}}}

\def\eps{\epsilon}

\def\J{{\mathcal{J}}}

\def\cE{{\mathcal{E}}}
\def\cC{{\mathcal{C}}}
\def\vp{{\varphi}}

\def\C{\Bbb C}
\def\N{\Bbb N}
\def\Z{\Bbb Z}

\def\R{{\Bbb{R}}}
\def\br#1{\left(#1\right)}

\def\HD{{\mathrm{dim_H}}}

\def\Mis{\mathrm{Mis}}

\begin{document}
\title{Precise asymptotics at the tip of the Mandelbrot set}

\maketitle

\begin{abstract}
	For the quadratic family $f_c(z)=z^2+c$, the only parameters in the
	Mandelbrot set $\M$ for which the  Julia  set $\J_c$ has Hausdorff dimension $1$ 
	are $c=0$ and $c=-2$.  
	Near $c=0$, Ruelle's  theory gives a real-analytic expansion of the dimension.
	The  tip $c=-2$ of $\M$, however, is a non-hyperbolic parameter and the
	dimension function $c\mapsto \HD(\J_c)$ is highly discontinuous there. We prove the sharp first-order asymptotic for the lower envelope of the Hausdorff dimension at the tip: If $c\in \M $ then $\HD(\J_c)$ lies  asymptotically above   $1+ \Omega \sqrt{|c+2|}$ with the Jaksztas constant $\Omega =\sqrt{\frac{2}{3}}\frac{1}{\pi\log 2}$. 
This is a surprisingly   precise contribution to the Yoccoz problem  about unfolding attractors. 

The proof develops a thermodynamic formalism for degenerating families
of box mappings.  At each scale, for parameters $c\to -2$, the 
induced dynamics  exhibit  a uniform property
of exponential tails, generating improved control  of their pressure functions. 
\end{abstract}

\section{Introduction}
\paragraph{Yoccoz problem  and lower  dimension envelope.}
The Hausdorff dimension of Julia sets is one of the most sensitive
global parameters  in holomorphic dynamics.  For hyperbolic rational
maps, Bowen's formula identifies the Hausdorff dimension as the zero
of a pressure function, and Ruelle proved  \cite{Ruelle_1982} that this dimension varies real-analytically in analytic hyperbolic families.  These results belong
to the classical uniformly hyperbolic theory.  They do not, however,
address what happens when the family fails to be hyperbolic  and the thermodynamic formalism itself has to be rebuilt on scales
which move with the parameter.

This paper studies precisely such a singular degeneration.  We consider
the quadratic family
$ f_c(z)=z^2+c$
near the Chebyshev parameter $c=-2$.  At this parameter, the Julia set
is the interval $[-2,2]$. Zdunik's theorem~\cite{Zdunik_1990} shows that, together with $c=0$, this is the only parameter in the Mandelbrot set for which the connected Julia set has Hausdorff dimension $1$. As hyperbolic sets admit continuations, Zdunik's theorem gives even more, $\HD(\J_c)$ can approach the singular value $1$ only for $c\in \M$ that  approach $0$ or $-2$.
The two cases are of  different nature. 

Near $c=0$, the
Julia set remains a quasicircle and Ruelle's perturbative hyperbolic theory  \cite{Ruelle_1982}  gives
the expansion
$$\HD(\J_c)
=1+\frac{|c|^2}{4\log 2}+o(|c|^2).$$

The situation near $c=-2$ is  different due to  the lack of hyperbolicity. For example,  parameters $c \in \M$ with $\HD(\J_c) = 2$ are dense in $\partial \M$ and hence accumulate on $-2$ by Shishikura, \cite{Shishikura_1998}.
On the other hand, Graczyk and Smirnov \cite{Graczyk_2008} showed that, restricting to uniformly summable parameters, the Hausdorff dimension varies continuously. In particular, there are parameters approaching $-2$ for which $\HD(\J_c)$ approaches $1$.  Thus the Yoccoz conformal problem to understand unfolding of
$[-2,2]$ into  fractal Julia sets $\J_c$ 
through properties  of the dimension  function  $c\in \M\mapsto \HD(\J_c)$ 
does not belong to any known  perturbation theory.  One needs to extract the  lower envelope from  $c\in \M\mapsto \HD(\J_c)$ near a singular tip of $\M$ through a new construction.

In the recent work \cite{Dobbs_2022}, the following estimate was obtained for all real parameters. 
There exist $\Omega', \eps_0 >0$  such that for all $c \in [-2, -2+\eps_0)$,
$$\HD(\J_c)\geq 1+ \Omega' \sqrt{|2+c|}.$$


\subparagraph{Main results.}
It is natural to ask, in the context of the Yoccoz problem cf.~\cite{Dobbs_2022}, whether one can find an optimal constant for the lower  dimension envelope  near $c=-2$.  Jaksztas \cite{Jaksztas_2023} proposed the constant $\Omega$  
taking limits of his hyperbolic estimates outside $\M$,
$$\Omega:=\sqrt{\frac{2}{3}}\frac{1}{\pi\log 2}.$$
%
\begin{theo}\label{theo:optimal}
For every  $\theta>0$, 
\begin{equation}\label{eqn:opt}
	\liminf_{c\to -2} \frac{\HD(\J_c)-1}{\sqrt{|c+2|}} = \Omega,
\end{equation}
where the limit is taken over parameters $c$ in the cusp 
$$\{-2 + x + iy : x\geq 0, |y| \leq x^{1+\theta}\}.$$ 
\end{theo}
In the Mandelbrot set, parameters near $-2$ are contained in
 the cusp  provided $\theta< 1/2$, see Section~\ref{sec:ManBound}. 
Theorem~\ref{theo:optimal} shows that the dimension function $c\mapsto \HD(\J_c)$ at the tip $c=-2$ is governed by the square root law as opposed to
 Ruelle's quadratic law at $c=0$.

We also prove that the constant $\Omega$ is attained along dynamically natural
parameters on the real line.  For $\kappa>0$, let
$$\Mis_\kappa := \{c \in \M : 
\forall {n\geq 1}, ~~~ |f_c^n(c)|\geq \kappa\},$$
the set of $\kappa$-Misiurewicz quadratic parameters, and
$\Mis:=\bigcup_{\kappa >0}  \Mis_\kappa$.
For $c$ close to $-2$ we have the following sharp estimate.
\begin{theo}\label{theo:upper}
$$
\lim_{c \to -2, c \in \Mis_\kappa} \frac{\HD(\J_{{c}})-1}{\sqrt{c+2}} =\Omega.$$
\end{theo}
The fixed-$\kappa$ condition is essential:
if $\kappa$ goes to $0$  with $c+2$ then
renormalization technique and Shishikura's theorem yield parameters
arbitrarily close to $-2$ with $\HD(\J_c)$ close to $2$.

We prove Theorem~\ref{theo:upper}  only for real parameters. 
This suffices since  all $\kappa$-Misiurewicz parameters in a small enough neighbourhood of $-2$ are real. 
\begin{prop}\label{thmMisReal}
Given $\kappa>0$, there exists $\delta>0$ such that 
 $$\{c\in \M: |c+2|<\delta\}\cap  \Mis_\kappa \subset \R. $$ 
\end{prop}

\paragraph{Heuristic  meaning of $\Omega$.}
The constant $\Omega$ has a simple geometric origin.  Write
$\epsilon=c+2$, with $\epsilon>0$ real.  The fixed point near $c=2$ is
$$ p(c)=2-\frac{\epsilon}{3}+O(\epsilon^2).$$
Pulling back the interval between $-p(c)$ and $c$ produces, through the
critical point, a transverse \emph{cross} of length
$$2\sqrt{\frac{2\epsilon}{3}}+O(\epsilon).$$
The factor $1/(2\pi)$ is the density at the critical point of the
absolutely continuous invariant probability measure of the Chebyshev
map $f_{-2}$, and $\log 2$ is its Lyapunov exponent.  The dimension
increase is therefore obtained by balancing the frequency of visits to
the critical scale with the size of the transverse crosses through $\J_c$.
This gives
$$\Omega= \sqrt{\frac{2}{3}}\frac{1}{\pi\log 2}.$$

For complex parameters in the cusp, the cross deforms and may separate, but
the separation is of higher order than $|\eps|^{1/2}$.
\paragraph{Uniform thermodynamical formalism for singular  families.}
The main difficulty is to translate heuristic ideas  about $\Omega$ into a uniform
thermodynamic formalism.  The usual hyperbolic perturbation theory does
not apply near $c=-2$.  The relevant branches of the dynamics are not
described by a fixed Markov partition, and the induced maps which capture
the geometry of $\J_c$ have combinatorics which degenerate as
$\epsilon=c+2\to 0$.  In particular, the number and depth of the essential
branches grow  as the parameter $c$ approaches the tip. 

Our principal technical contribution is building a thermodynamic formalism for
this degenerating family of induced box mappings.  For each small
parameter $\epsilon$ we construct a generalized Cantor repeller
$$ \phi_\epsilon:\bigcup_j D_{j,\epsilon}\longrightarrow D_{\epsilon}$$
obtained from first-entry and first-return map to dynamically defined
Yoccoz pieces near the critical point.  Although the family is not
compact, 
the induced box maps have
uniform exponential tails and are extensible to a common larger domain.
This leads to  a family-level no-escape mechanism: the  pressure and Lyapunov mass cannot  escape  into increasingly remote branches when $c\mapsto -2$.
Consequently, the systems behave similarly to compact ones obtained through truncation to a large but finite number of branches. 
 The discarded branches
have a uniformly negligible pressure contribution. 
This allows us to compare the corresponding  pressure functions
$ P_\epsilon(t) $ uniformly for $t$ in a neighbourhood of $t=1$,
$$    P_\eps(t)\longrightarrow P_0(t).$$
However, the comparison is imprecise and depends on the inducing domain. Letting the depth of the inducing domain increase generates a sequence of families of induced maps for which we obtain increasingly accurate estimates. 
This gives us the
control needed to locate the zero of $P_\epsilon(t)$ up to the first order.


\def\svgwidth{\textwidth}
\begin{figure}
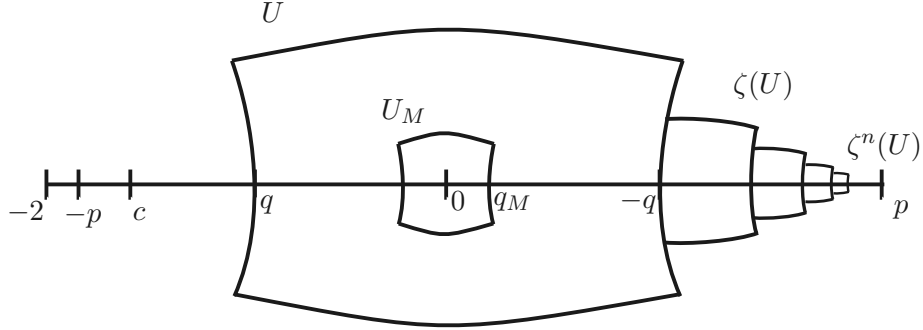
    \caption{For real parameters $c = -2+\eps$, the fixed points $p \approx 2 - \frac{\eps}3$ and $q \approx -1 + \frac{\eps}3$ are real. The fundamental domain $U$, defined in \S\ref{sec:funddom}, is constructed using dynamical rays and arcs of an ellipse, as is the deep inducing domain $U_M$ defined in \S\ref{secumqm}. As $M\to \infty$, the diameter of $U_M$ tends to $0$. The map $f$ has an inverse branch $\zeta$ with fixed point $p$.  }
  \end{figure}

\paragraph{Related formalisms and settings.}
The main ergodic theme  of the paper  lies at the intersection of several thermodynamic formalisms but outside their standard compactness regimes. A  compact  case was studied by Rugh for parameter dependent  random conformal
repellers under uniform expansion condition~\cite{Rugh_2008}.
Buzzi, Crovisier and Sarig relied on  countable state thermodynamic formalism~\cite{SARIG_1999}  to advance  symbolic models in non-uniformly hyperbolic dynamics \cite{CrovSar2022MME}.  Their continuity defect theory allows both the map and the invariant measure to vary~\cite{BuzCrovSar2022Lyap}.

In our setting,  the
coding bifurcates with the parameter. The alphabet and geometry vary and
the inducing scheme degenerates  as $\eps\to0$. As in Ruelle's case, the goal is an explicit  asymptotic near the root of the pressure. The key observation and the first step is to replace compactness by  uniform exponential tails to prevent dissipation of the  pressure towards degenerating  branches. It is worth to emphasize that our estimates hold for every $c\in \M$ sufficiently close to $-2$ rather than only  on a set obtained by  the parameter exclusion or for a typical subfamily.

\paragraph{The lower envelope and sharpness.}
\def\svgwidth{11cm}
\begin{figure}
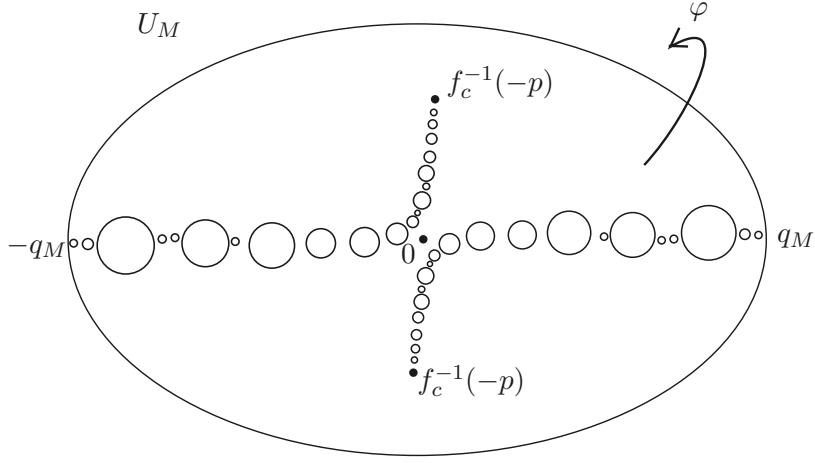
    \caption{The induced map $\vp$, constructed in \S\ref{sec:definevp}, to the deep inducing domain $U_M$ has two distinguished collections of infinitely many branches, those close to the real axis almost joining $\pm q_M$ and those joining the two preimages by $f_c$ of $-p$, close to the imaginary axis. The diagram is schematic; the distance between the preimages of $-p$ is approximately $2\sqrt{\frac{2}{3}} \sqrt{|\eps|}$.}
    \label{fig:umphi}
  \end{figure}
The induced box mapping (see Figure~\ref{fig:umphi}) has two geometrically distinguished families of
branches. The first is asymptotically horizontal and encodes the limiting
interval dynamics. The second is transverse and is responsible for the
first-order increase in dimension, its total transverse size is
$$
2\sqrt{\frac{2}{3}}\,|\eps|^{1/2}
+o\left(|\eps|^{1/2}\right).$$
For an inducing domain $U_M$ of depth $M$, this geometry yields
\begin{equation}\label{equ:pres_eps}
P_{\eps}(1) \geq\frac{1}{|q_M(c)|}\sqrt{\frac{2}{3}}\,|\eps|^{1/2}
\left(1+o(1)\right),
\end{equation}
where the errors are understood in the iterated limit
$\eps\to0$ with $M$ fixed, followed by $M\to\infty$. The points $q_M=q_M(c) \in \partial U_M$, defined in \S\ref{secumqm}, satisfy $\lim_{M\to \infty}\lim_{c \to -2} q_M(c) = 0$. 

For fixed $M$, the uniform thermodynamic formalism compares the
pressure slope (whose strong dependence on $M$ is suppressed from the notation) with that of the Chebyshev return map
uniformly for $t$ in a neighbourhood of $1$. By a direct calculation, 
$$ \lim_{M\to\infty}-\frac{q_M(-2)}{\pi}P'_{0}(1)=\log 2,$$
and consequently, in the same sense as (\ref{equ:pres_eps}), 
$$ P_\eps'(t)    =-\frac{\pi\log 2}{q_M(-2)}\left(1+o(1)\right).$$

The induced map defines a generalized Cantor repeller $\Lambda_{\eps}\subset \J_c$.  Bowen's formula gives
$P_{\eps}^{-1}(0)=\HD(\Lambda_{\eps})$.
By the mean value theorem,
$$ \HD(\Lambda_{\eps})-1 =
\frac{P_{\eps}(1)}
{-P'_{\eps}(\xi_{\eps})}$$
for some $\xi_{\eps}\in(1,\HD(\Lambda_{\eps}))$.  Since the inducing scale  factors  $|q_M(c)|$ and $q_M(-2)$ are asymptotically equal and  occur in the pressure and its
slope,
$$ \HD(\J_c)-1 \geq \sqrt{\frac{2}{3}}\frac{1}{\pi\log 2}~ {|\eps|}^{1/2} 
\left(1+o(1)\right).$$

The reverse inequality for the lower envelope follows from
Theorem~\ref{theo:upper}. For real $\kappa$-Misiurewicz parameters,
the postcritical orbit remains a definite distance from the critical
point, so the recurrent central obstruction is absent and the induced
repeller has full Hausdorff dimension in $\J_c$. This gives the matching
upper estimate and the sharp asymptotic in Theorem~\ref{theo:upper}.


\paragraph{Collet-Eckmann estimates.} A  polynomial $f_c(z)=z^2+c$ is Collet-Eckmann~\cite{Collet_1983} if
$$ \liminf_{n\to\infty}\frac1n  \log |(f_c^n)'(c)|>0.$$
In~\cite{Dobbs_2022}, it was  proved that on a set of {\em real} Collet-Eckmann
parameters having $-2$ as a one-sided Lebesgue density point, there exists
$C>1$ such that
\begin{equation}\label{equ:CE-bounds}
	1+C^{-1}\sqrt{c+2}  \leq \HD(\J_c)  \leq 1+C|\log(c+2)|\sqrt{c+2}.
\end{equation}
The upper estimate of (\ref{equ:CE-bounds}) is based on Jones's $\beta$-numbers and their
non-uniform extension~\cite{Jones_1990,Graczyk_2022}.  That geometric $L^2$ method is
robust 
but it does not capture  the first order pressure contribution needed for an explicit
sharp constant.  The present inducing and pressure method is complementary.
It gives a deterministic sharp lower estimate throughout the cusp and isolates
the universal square-root contribution responsible for $\Omega$.


\paragraph{Organization of the paper.}
The paper is organized in the following way.  Sections~\ref{sec:Initest} and~\ref{sec:Geom} give the elementary
expansion, continuation and geometric estimates near the Chebyshev
parameter. The technique is based on  holomorphic motions but overall direct and elementary estimates are privileged.  Sections~\ref{sec:CanRep}--\ref{secTDF} develop the pressure formalism for convergent
families of generalized Cantor repellers with uniform exponential tails.
Sections~\ref{sec:Lyap} and~\ref{sec:mis} compute the limiting pressure slope and prove the sharp
estimate for real Misiurewicz parameters.  Section~\ref{sec:tails} constructs the
induced box mappings needed in the cusp and proves their uniform
exponential tails and linkage properties.  Section~\ref{sec:pres} completes the
pressure estimate and proves the sharp lower-envelope asymptotic.




\section{Initial estimates}\label{sec:Initest}

We consider $f_c : z \mapsto z^2 + c$, $c = -2 + \eps$, with $\eps \in \C$ small.

Denote by $\zeta = \zeta_c$ the branch of $f_c^{-1}$, with a fixed point 
 $$p = p(c)  = 2 - \eps/3 + O(\eps^2)$$ 
 near $2$, defined on the right half-plane containing $c$ in its boundary. 
We denote by $q = q(c) = -1+\eps/3 + O(\eps^2)$ the other fixed point of $f_c$. 

We write
$$ 
A~\lesssim~B
$$
if there is a constant $C >1$ for which the expressions $A, B$ satisfy the statement $A \leq CB$ for all small $\eps$. We define $\grtsim$ analogously. If both statements hold, we write $A~\sim~B$. 

We shall occasionally use the usual big-$O$ notation.
\subsection{Expansion estimates}

On $\{z : \Re(z)> -1.2\}$, $\zeta$  is univalent and a strict contraction, with $|D\zeta| < 3/5$ provided $\eps$ is small. Consequently, on the ball $B(2, 3.2)$, iterates $\zeta^n$ have uniformly bounded distortion, so 
\begin{equation} \label{eqn:zetadist}
    |D\zeta^n(z)| ~\sim~
|p - \zeta^n(z)| ~\sim~ |Df_c(p)|^{-n}
\end{equation}
for any $z$ in the ball. 
Hence, if $z \in B(0, 1.1)$, $C>0$ and $k$ is minimal for which $|\zeta^k(z)-p| < C|\eps|$, then 
\begin{equation} \label{eqn:epslog4}
    |2p|^k ~\sim~  |\eps|^{-1}.
\end{equation}

The sequence of sets $\zeta^n(B(0,1.1))$ form a chain of topological discs converging to $p$. 

For $c=-2$, $f_{-2}(B(0,1.1)) = B(-2, 1.21)$, which compactly contains $-\zeta_c^n(B(0,1.1))$ for all $n\geq 1$. In particular, if $z\in B(0,1.1)$ returns to $B(0,1.1)$ with (minimal) return time $n\geq2$, then $f_c^{2}(z) \in \zeta_c^{n-2}(B(0,1.1)).$

\begin{lem} \label{lem:sqrtratio}
    For $|c+2| = |\eps|$ small enough, 
    if $10|\eps|^{1/2} \leq |w| \leq 1/4$ and if $f_c(y) = -f^2(w)$, then $|y|>|w|$ and
    $$\left| \frac{w}{y} Df_c(f_c(w))\right| > 3/2.$$
\end{lem}

\begin{proof}
    As $|w|\leq1/4$, 
    $$|f_c(w) + 2| \leq 1/16 + |\eps| < 1/15$$
    while crudely, on $B(-2, 1/15)$, $4 - 1/4 \leq |Df_c| < 4 + 1/4.$
    Therefore,
    $$
    \frac{15}{4} \leq \left|\frac{f^2_c(w) - p}{f_c(w) + p} \right| \leq \frac{17}4.$$
    Meanwhile, $|f_c(w)-c| \geq 100\eps  \geq 100|p+c|$, so
    $$
    \frac{14}{4} \leq \left|\frac{f^2_c(w) + c}{f_c(w) - c} \right| \leq \frac{18}4.$$
    If $f_c(y) = -f^2_c(w)$, consequently,
    $\sqrt{7/2} \leq |y/w| \leq \sqrt{9/2}$ and,
    for this $y$, 
    $$\frac{|Df^2_c(w)|}{|Df_c(y)| }  \geq \frac{15/4}{\sqrt{9/2}} =\frac32 \cdot \frac54 \sqrt{\frac89}> 3/2.$$
\end{proof}

The following lemma is not especially profound, but exponential growth of the derivative simplifies many subsequent considerations. 
\begin{lem}  \label{lem:expandingreturns}
    For $|c+2| = |\eps|$ small enough, 
    if $z \notin B(0, \eps^{2/3})$ and $n$ is a first return time of $z$ to $B(0,1.1)$ then $|Df_c^n(z)| > 2^{n/2}.$
\end{lem}
\begin{proof}
    If $|z|\geq 3/4$, this is also true for iterates $f^j_c(z)$, $j <n$, therefore $|Df_c(f^j_c(z))| \geq 3/2$ and $|Df^n_c(z)| \geq (3/2)^n$. 

    If $2/5 \leq |z| \leq 3/4$, then $|f_c(z)| \geq 2 - 9/16 - |\eps|$ and $|Df_c(z)| \geq 4/5$, 
    so $|Df^2_c(z)| \geq \frac45(23/8 - 2|\eps|) > 2$. 

    If $1/4 \leq |z| \leq 2/5$, then $|f_c(z)| \geq 46/25 - |\eps| > \sqrt{2+1.1}$, so $n\geq 3$ and 
     $$|Df^3_c(z)| > \frac12 2 \left(\frac{46}{25} - |\eps|\right) 2.2 > 3 > 2^{3/2}.$$ 

    If $10|\eps|^{1/2}  \leq |z| \leq 1/4$, we define an auxiliary sequence $y_1, \ldots$ by $$f_c(y_j) = -f_c^{j+1}(z),$$ 
    choosing arbitrarily between $\pm y_j$. Set $y_0 := z$. 
    By the definition and the Chain Rule, for $k\geq 2$, 
    \begin{equation}\label{eqn:telprod}
    \left| \frac{Df_c^k(y_j)}{Df_c^{k-1}(y_{j+1})} \right| = 
    \left| \frac{y_j}{y_{j+1}} Df_c(f_c(y_j))\right|.
    \end{equation}
    Let $k_0$ be minimal for which $|y_{k_0}| \geq 1/4.$ From  \eqref{eqn:telprod} and Lemma~\ref{lem:sqrtratio},
    $$
    \left| \frac{Df_c^{n}(z)}{Df_c^{n-k_0}(y_{k_0})} \right| 
    \geq (3/2)^{k_0}.$$
    Combine with the estimate for $|z| \geq 1/4$ to complete this case.

    If $|z| \leq 10|\eps|^{1/2}$ then $|f_c(z) + p| \leq 200|\eps|$. Let $n_0$  be minimal with 
    $$|f^{n_0}_c(z) - p| 
    \geq |\eps|
    $$
    and let $n_1 = n-n_0$. 
    Using \eqref{eqn:zetadist} and \eqref{eqn:epslog4},
     $$|Df_c^{n_1}(f^{n_0}_c(z))| ~\sim~ |\eps|^{-1} ~\sim~|2p|^{n_1},$$ 
     while $|z| \geq |\eps|^{2/3}$, so 
     $$\left|Df_c(z)Df_c^{n_1}(f^{n_0}_c(z))\right| ~\grtsim~ |\eps|^{-1/3} ~\sim~ |2p|^{n_1/3} > 2^{n_1/2} .$$
    Iterates from $1$ to $n_0$ improve the derivative estimate. 
\end{proof}

    The expression 
    $$
    T_n(z) := \sum_{j=1}^n\frac{1}{Df_c^j(z)}
    $$
    emerges naturally from the Implicit Function Theorem, see the next subsection. We have elementary estimates: 

    \begin{lem}
    If $|f^j_c(z)| > 0.9$ for $j = 0, \ldots, n-1$, 
    \begin{equation}\label{eqnsmallspeed}
        |T_n(z)| < \frac54
\end{equation}
    and if, moreover, $|f^j_c(z)| > 3/2$ for $j=0, \ldots, n-2$, then 
    \begin{equation}\label{eqnedgespeed}
        |T_n(z)| < \frac45.
\end{equation}
    If $|f^j_c(z)| > 0.9$ for $j = 1, \ldots, n-1$, 
    \begin{equation}\label{eqnreturnspeed}
        |T_n(z)| <  \frac94\frac{1}{|Df_c(z)|}.
\end{equation}
    \end{lem}

\begin{lem} \label{lemSplitreturns}
    There exists a universal constant $C>1$ such that the following holds. 
     If $|f^j_{c_0}(z)| > |c_0+2|^{2/3}$ for $0 \leq j \leq n-1$, 
    $$
    |T_n(z)| \leq   C \left(\inf_{0\leq j <n} |f^j_{c_0}(z)|\right)^{-1}.$$ 
     If $|f^j_{c_0}(z)| > |c_0+2|^{2/3}$ for $0 \leq j$, then the series $T_\infty(z)$ converges absolutely. 
\end{lem}
\begin{proof}
    Let 
    $$
    \hat{T}_n(z) := \sum_{j=1}^n\frac{1}{|Df_c^j(z)|}
    $$
    and let $n_k$ be the $k$th return time of $z$ to $B(0,1.1).$
    Then 
    $$
    \hat{T}_{n_{k+1}}(z) - 
    \hat{T}_{n_{k}}(z)  = \frac{1}{|Df_c^{n_k}(z)|} \hat{T}_{n_{k+1}-n_k}(f_c^{n_k}(z)).$$
    As $|Df_c(f^j(z))| >2$ for $j = n_k+1, \ldots, n_{k+1}-1,$
    $$\hat{T}_{n_{k+1}-n_k}(f_c^{n_k}(z)) \leq \frac{2}{|Df_c(f_c^{n_k}(z))|}.$$
    By Lemma~\ref{lem:expandingreturns}, ${|Df^{n_k}_c(z)|}^{-1} \leq 2^{-n_k/2}$. The result follows, one can take $C = 2 \sum_{k=1}^\infty 2^{-k/2}$. 
\end{proof}
\begin{lem} \label{lem:thirdquarter}
    There exists $\delta>0$ such that, for $|\eps|$ small enough, if $z \in B(p,\delta)$ and $n$ is a first entry time of $z$ to $B(0,1.1)$ then 
    $$
    \left| \frac 13 - T_n(z) \right| < 1/100.$$
\end{lem}
\begin{proof}
    This holds since $|Df_c(p)|$ is close to $4$ for small $|\eps|$, and $\sum_j 4^{-j} = 1/3$. 
\end{proof}

\subsection{Implicit function theorem and continuations}
    Applying the implicit function theorem to $f_c(z) - z$, we obtain 
    for the fixed points $p \approx 2$ and $q \approx -1$,  
    $$
    \frac{dp}{dc} = \frac{-1}{2p-1} \approx \frac{-1}3$$
    and 
    $$
    \frac{dq}{dc} = \frac{-1}{2q-1} \approx \frac13.$$
    
    Let $\rho$ be holomorphic, defined on a neighbourhood of some parameter $c_0$, and let $n\geq 1$.  Suppose that $f^n_{c_0}(z_0) = \rho(c_0)$. 
Let $G(z,c) = f^n_{c}(z) - \rho(c).$ 
Applying the implicit function theorem (if $Df^n_{c_0}(z_0) \ne 0$), we find a holomorphic function $h$ defined on a neighbourhood of $c_0$ with $h(c_0) = z_0$ and $G(h(c), c) = 0$. Moreover, 
    \begin{equation} \label{eqngenspeed}
        h'(c) = \frac{\rho'(c)}{Df^n_c(h(c))} - T_n(h(c)).
    \end{equation}



    We can also look at  perturbations of infinite orbits or, indeed, forward invariant sets. 
    The existence of \emph{holomorphic motions} is well known \cite{Jonsson_1998, Ruelle_2014, Ma__1983}. 
    We call a forward-invariant compact set on which $|Df_c^N|>2$ for some $N\geq1$ a \emph{repeller} for $f_c$. The final statement of the following can be obtained taking limits of \eqref{eqngenspeed}, with $\rho$ the constant function equal to $f^n_c(h(c))$. 
    \begin{lem} \label{lemHolMot}
        Let $K\subset \C$ be a repeller for $f_{c_0}$. There is an $r>0$ and a \emph{holomorphic motion} 
        $h : B(c_0,r) \times K \to \C$
        with $h(\cdot,z)$ holomorphic and $h(c, \cdot) :K \to h(c,K) = K_{c}$ a homeomorphism.
        Moreover, $K_{c}$ is hyperbolic for $f_{c}$ and 
        $$
        f_{c}\circ h(c,\cdot) = h(c,\cdot) \circ f_{c_0}.$$
        
        The partial derivative with respect to parameter satisfies
        $$\partial_1h(c,z) = -T_\infty(h(c,z)).$$ 
    \end{lem}
    \begin{coro}\label{cor:holospeed}
        If $\theta \in (0, 1/12)$, if $c_0 = -2+\eps_0 \ne -2$ and $\eps_0$ is small then, letting 
        $$
        K = \{z \in \J_{c_0} : |f^n_{c_0}(z)| \geq |\eps_0|^{1/2+\theta} \text{ for all } n\geq0\},$$
        the lemma holds with $r = |\eps_0|^{1 + 4\theta}$ and, for $c \in B(c_0,r)$, for all $z \in K$, 
        $$|\partial_1h(c,z)| \leq |\eps_0|^{-1/2 -2\theta}.$$ 
    \end{coro}
    \begin{proof}
        Let $r'$ be the radius of the maximal ball, centred at $c_0$, on which $h$ is defined and the derivative estimate holds. By Lemmas \ref{lemHolMot} and \ref{lemSplitreturns}, $r'>0$. Suppose $r'<r$ and $c \in \overline{B(c_0,r')}$. Then 
        $$|h(c,z)-z| \leq C|\eps_0|^{-1/2-2\theta} r' < C|\eps_0|^{1/2 + 2\theta}$$
        so
        $|h(c,z)| \geq \frac12 |\eps_0|^{1/2 + \theta}$ for all $z\in K$. 
        By Lemma~\ref{lemSplitreturns},
        $|T_\infty(h(c,z))|$ is bounded by $2C |\eps_0|^{-1/2 - \theta}$. Hence, on some larger ball, $h$ exists and the derivative estimate holds, contradicting the definition of $r'$.
    \end{proof}

\section{Geometric estimates} \label{sec:Geom}

We continue to consider $f = f_c : z \mapsto z^2 +c$, where $c = -2 + \eps$, with $\eps$ small and complex.
We denote one-dimensional Lebesgue measure on $\R$ by $\Leb$ and denote by $\diam(U)$ the diameter of a set $U$. 

\subsection{Ellipses containing the Julia set} \label{secEA}
In the extremal case $\eps = 0$, the Julia set $\J_c$ of $f$ is the real interval $[-2,2]$. 
We define a family of ellipses $\cE_a$, $a>0$, 
$$\cE_a = \{z : |z-2| + |z + 2| < 4+a\}.$$
The boundaries $\partial \cE_a$ of these ellipses are Green's curves and form an invariant family. They are the images of circles centred on 0 by the Joukowski map $z \mapsto z + 1/z$. 
\begin{lem}[\cite{Bara_ski_1998}]
If $a > 0$, then 
    \begin{equation}\label{eqn:greens}
        f_{-2}(\partial \cE_a) = \partial \cE_{\frac{(a+4)^2}2 - 8}.
    \end{equation}
\end{lem}
\begin{proof}
Indeed, let $z = re^{i\theta} \in \partial \cE_a$, then 
\begin{eqnarray*}
    (4+a)^2 & = & |z-2|^2 + |z+2|^2+2|z-2||z+2| \\
& = & 2r^2 +8 - 4r\cos \theta + 4r \cos \theta + 2|z-2||z+2|.
\end{eqnarray*}
    so $
    (4+a)^2/2 - 4 =
    r^2 + |z-2||z+2|.  $
On the other hand,
\[
|f_{-2}(z) + 2| + |f_{-2}(z)-2| = r^2 + |z-2||z+2|,
\]
    showing \eqref{eqn:greens}.
\end{proof}
    
\begin{lem} \label{lemJellipse}
    If $0 < |\eps| \leq a < 2$, 
    then $f_c(\partial \cE_a) \cap \overline{\cE_a} = \emptyset.$ 
    Moreover,
    the Julia set $\J_c$ is contained in  $\cE_{|\eps|}$ and 
\begin{equation}\label{eqn:imband}
    \Im(\J_c) \subset [-2\sqrt{|\eps|}, 2\sqrt{|\eps|}].
\end{equation}
\end{lem}
\begin{proof}
The shortest path between the boundaries of two such ellipses is along the real axis, so 
$$\dist { \cE_a , \partial f_{-2}(\cE_a)} =  \frac{\frac{(a+4)^2}2 - 8 -a}{2} = 3a/2 + a^2/4.
$$
Consequently, if $|\eps| \leq a  < 2$, then 
$
f_{-2+\eps}(\cE_a) = \eps + f_{-2}(\cE_a)$ contains $\cE_a$, whence $\J_c$ is contained in  $\cE_{a}$.
    Noting that $|2\sqrt{a} i \pm 2| = 2\sqrt{1+a} > 2 + a/2$, no point with imaginary part $2\sqrt{a}$ lies in $\cE_a$. Hence
\eqref{eqn:imband} holds.
\end{proof}

\subsection{Fundamental domain}\label{sec:funddom}

The two geometric rays $\{\arg(z) = \frac{2\pi}{3}, |z|\geq 1\}$ and $\{\arg(z) = \frac{4\pi}3, |z|\geq 1\}$ are mapped to each other by $z \mapsto z^2$. Their image under the Joukowski map $z \mapsto z + \frac1z$ is an $f_{-2}$-invariant Jordan curve $\Gamma_0$ containing $q$ and $\infty$ (viewed on the Riemann sphere). 
We note that $\Gamma_0$ is orthogonal to the family of ellipses $\{\cE_a\}_a$ with foci $\pm2$ (defined in \S\ref{secEA}).  


On $B(\Gamma_0, \rho)$, $|Df_{-2}| > 1.6$ if $\rho \in (0, 0.2)$. 
Hence, if $|\eps| < 0.1$ then 
$$\eps + f_{-2}(B(\Gamma_0, 2|\eps|)) \supset B(\Gamma_0, 2|\eps|)$$
and there is an $f_c$-forward-invariant curve, with $c = -2+\eps$,
$$
\Gamma_\eps \subset B(\Gamma_0, 2|\eps|).$$

\begin{remark} If $c\notin \R$, then $\Gamma_\eps$ will spiral in to $q$. Note, the spiralling is negligible since $\arg(q)$ is small; the radius decreases by a factor of something like $2^{-6\pi/\Im(\eps)}$ each full rotation.  If $c \in \M$, then $\Gamma_\eps$ is the union of the two dynamical rays landing at $q$, the image of the geometric rays above under the canonical Riemann map from the exterior of the unit disc to the exterior of the Julia set. We allow $c$ to leave the Mandelbrot set, resulting in disconnected Julia sets. 
    \end{remark}

The curves $\pm \Gamma_\eps$ divide the solid ellipse $\cE_{a}$ into simply connected regions. If $a$ is fixed and small and if $|\eps|$  is small enough, the  region containing the critical point $0$  is compactly contained in $B(0,1.1)$. 
We denote it by $U=U(c)$, a \emph{fundamental domain} of $f_c$. 
Recalling Lemma~\ref{lemJellipse}, $\J_c \subset \cE_a$. Every point of $\J_c$ enters ${U} \cup \{p, q\}$ under iteration.

\subsection{Mandelbrot set bound} \label{sec:ManBound}
Let $x,y \in \R$, $0 \leq y <x$, and consider
 $\eps = x+iy$ small. By Lemma~\ref{lemJellipse}, $\J_c \subset \cE_{2x}$ and \eqref{eqn:imband} bounds its imaginary part by $2\sqrt{2x}$.  The sets $\zeta^n(\cE_{2x} \cap B(0,1.1))$ form a chain of topological discs leading towards $p$. If $y$ were equal to $0$, the chain would be contained in a cone  $\cC$ with vertex $2$ of narrow aperture  $C\sqrt{x}$, with $C>0$ a universal constant. 

    Applying Lemma~\ref{lem:thirdquarter} and equation~\eqref{eqngenspeed}, the chain near $2$ is contained in $B(\cC, y/2)$ and so has imaginary part bounded by $y/2 + 10Cx^{3/2}$ in $B(2, 5x)$. 
    If $y > 20Cx^{3/2}$, then $2 - (x+iy)$ is not contained in the chain and so is not in the filled Julia set. Hence $-2 + x+iy \notin \M$. 
 
\subsection{A smaller inducing domain $U_M$} \label{secumqm}
For $M>0$, the collection of curves $f_c^{-M}(\Gamma_\eps)$ divide  $\cE_{1/M}$ into regions. Denote the one containing $0$ by $U_M(c)$. It exists, and 
 $f^j_c(\partial U_M(c)) \cap U(c) = \emptyset$ for all $j\geq 1$, 
if $\eps$ is small enough.

Denote by $\pm q_M$ the preimages of $q$ in $f_c^{-M}(\Gamma_\eps)$, so $q_M=q_M(c)  \in \partial U_M(c)$. We choose $q_M$ to have positive real part when $\eps$ is small. 
Note that, if $c=-2$ and $1<j<  M-2$,
\begin{equation}\label{eqnqmorbit}
    f_{-2}(q_M) < q <  0 < q_M < -q = f^{M-1}_{-2}(q_M) < f^{j+1}_{-2}(q_M) < f^j_{-2}(q_M) < 2.
\end{equation}

    Let $K_M(c)$ denote the set of points in $\J_c$ which never enter $U_M(c)$. 
    As $K_M(-2)$ is a repeller for $f_{-2}$, Lemma~\ref{lemHolMot} gives an $\eps_0>0$ and a holomorphic motion $h: B(-2, \eps_0) \times K_M(-2) \to \C$, 
    $$
    f_c \circ h(c,\cdot) = h(c,\cdot) \circ f_{-2}.$$
    Moreover, 
        $$|\partial_1h(c,z)| = |T_\infty(h(c,z))| $$ 
        is uniformly bounded by (taking the limit of) the first part of Lemma~\ref{lemSplitreturns}. As a consequence, we obtain:
    
        \begin{lem} \label{lemKMspeed}
            Given $M$, there exist $\eps_1>0$, $C>1$ such that, for $c=-2+\eps$, $|\eps|<\eps_1$, 
            $$
            \frac1n \left|\log |Df^n_c(h(c,z))| 
            -
            \log |Df^n_{-2}(z)|  \right| \leq C|\eps|$$
            for all $z \in K_M(-2)$ and $n\geq 1$.
        \end{lem}

        For $M$ large, the modulus of the annulus $U(c)\setminus U_M(c)$ is large. 
        \begin{lem}\label{lem:UMext}
            For $M$ large enough, if an open set $\hat W$ is mapped by $f_c^n$ biholomorphically to $U(c)$ and $W \subset \hat W$ is mapped to $U_M(c)$, 
            then
            $$
            B(W, 2\diam(W)) \subset \hat W.
            $$
        \end{lem}
        \subsection{Nice sets}
        
\begin{defi}
    An open set $W$ is called 
    \emph{nice} (or \emph{regularly returning}) 
    if
    $f^n(\partial W) \cap W = \emptyset$ for all $n\geq0$. 
\end{defi}
Nice sets have the useful
    property, preimages are \emph{nested or disjoint}: If $A$ is a connected component of $f^{-n}(W)$ and $B$ of $f^{-m}(W)$, $0\leq n \leq m$, then $B \subset A$ or $A \cap B = \emptyset.$ 

    \begin{lem} \label{lemUareNice}
        Both $U(c)$ and $U_M(c)$ are nice sets. 
    \end{lem}
    \begin{proof}
        Noting the first statement of Lemma~\ref{lemJellipse} and forward-invariance of the dynamical rays through $\pm q$, this follows from construction.
    \end{proof}

    In particular, if $W$ is $U$ or $U_M$ and $\vp$ is the first return map to $W$, one 
    connected component of its domain may contain $0$; it is called the \emph{central} branch domain and is mapped with degree 2 by $\vp$ onto $W$. Each non-central component  is mapped by $\vp$ biholomorphically onto $W$.

\subsection{$\kappa$-Misiurewicz maps are real}
    Given $\kappa>0$, we choose $M\geq 1$ such that, for $c=-2$, $U_M \subset B(0, \kappa/2)$. For $c$ close enough to $-2$,  $U_M(c) \subset B(0, \kappa)$. 
    Any $\kappa$-Misiurewicz parameter $c$ must lie in $K_M(c)$. 

    Take $N\geq 1$ large enough that $|T_\infty(h(c,z))| < 2^N$ for all $z \in K_M(-2)$ and $c$ close to $-2$.
    Note that 
    $$T_\infty(h(c,z)) = T_n(h(c,z)) + \frac{1}{Df^n_c(h(c,z))} T_\infty(f^n_c(h(c,z))).$$
    Consider a neighbourhood $W = B(\pm 2, \delta)$ of $\pm2$ on which, for all $c$ near $-2$, the first entry time of $z \in W$ to $B(0,1.1)$ is at least $20N$. 
    For $h(c,z) \in W$, let $n$ be the first entry time.  The second term on the right-hand side is bounded by $2^{-10N}2^{N} \leq 2^{-9} < 1/100$.  Apply Lemma~\ref{lem:thirdquarter} to bound $T_n(h(c,z))$. Then, 
    \begin{equation}\label{eqn:imspeed}
    |\partial_1 h(c,z)| = |T_\infty(h(c,z))| < 1/2.
    \end{equation}
    Hence, if $z \in K_M(-2) \cap B(-2, \delta/2)$ say, then $h(c,z) \in W$ for all $c$ near $-2$, while $h(-2+x,z) \in \R$ if $x\in \R$. By the speed estimate~\eqref{eqn:imspeed}, 
    \begin{equation}\label{eqn:IMspeed2}
    |\Im(h(c,z))| \leq |\Im(c)|/2.
    \end{equation}
    Let $c = -2+x+iy$, $x,y \in \R$ small, $y \ne0$. Then
    $K_M(-2+x) \subset \R$. 
    By~\eqref{eqn:IMspeed2},
    $$B(-2, \delta/2) \cap K_M(c) \subset \overline{B(\R, |y|/2)}.$$ 
    As $|y|>|y|/2$, $-2+x+iy \notin K_M(c)$, 
    so $c$ is not a $\kappa$-Misiurewicz parameter. 
    This proves Proposition~\ref{thmMisReal}.
    \subsection{Exponential tails}
    
\begin{defi}
We say that a map $g :  X \to \C $ has \emph{exponential tails} if $X$ is bounded and there are $C, \kappa >0$ such that, for $s \in \R$,
$$
\sum \diam(U) < C e^{-s\kappa},$$ 
where the sum is over connected components $U$ of the domain $X$ with diameter $\diam(U) < e^{-s}$.
\end{defi}
It is equivalent, in the definition, to consider the same sum over components with $e^{-s-T} \leq \diam(U) < e^{-s}$, which can be useful in proofs, with $T>0$ fixed.

A family of maps $(g_\eps)_{\eps \in A}$ has \emph{uniform exponential tails} if all $g_\eps$ have exponential tails, with constants $C,\kappa$ independent of $\eps \in A$.

    \subsection{First entry maps $F_M$ to $U_M(c)$ have uniform exponential tails}
        By Lemma~\ref{lem:expandingreturns}, the dynamics of $f_c$ is expanding outside $U_M(c)$. 
        Denote by $F_M$ the first entry map under $f_c$ to $U_M(c)$ (for legibility, we omit from notation the dependence of $F_M$ on $c$). As $U_M(c)$ is a nice set, preimages of $U_M(c)$ are nested or disjoint, and the first entry time is locally constant on connected components of the domain of $F_M$. 

        \begin{lem} \label{lemFmTails}
            There exists $\eps_1>0$ such that
            the family of maps 
            $$\{F_M\}_{c \in B(-2,\eps_1)}$$
            has uniform exponential tails.  
        \end{lem}
        \begin{proof}
        For each $w \in K_M(c)$ and $n_w\geq 0$ with $f^{n_w}(w) = q_M$, there is a unique branch domain $W_w = W_w(c)$ of  $F_M$ with $w \in \partial W_w$, and vice versa. By bounded distortion (for example, due to expansion) $\diam(W_w) ~\sim~ |Df^{n_W}_c(w)|^{-1}$. 

        Let $\delta >0$.
        By Lemma~\ref{lemKMspeed}, there exist $C>1$ and $\eps_1>0$ with $\eps_1C < \delta$ such that
        $$
        \frac1{n_W} \log(\diam(W_w(c)))$$ 
        varies by at most $\delta$ on $B(-2,\eps_1)$. Consequently, 
        it suffices to show that $F_M$ has exponential tails when $c=-2$, 
        
        Let us sketch the (well-known) case $c=-2$. 
        If $W$ is a branch domain of $F_M$, then $\diam(W)$ is comparable with $\diam(W\cap \R)$. If $\diam(W\cap \R) < 4^{-N}\diam(U_M(-2))$, then the entry time on $W$ is at least $N+1$. Thus $$W \cap \R \subset E_N,$$ 
        where $E_N\subset [-2,2]$ is the set of real points in $[-2,2]$ that have not entered $U_M$ in the first $N$ iterates. 
        Hence, the sum of $\diam(W)$ over such branch domains is at most comparable with $\Leb(E_N)$. 
        By~\eqref{eqnqmorbit} and induction, a connected component $V$ of $E_N$ gets mapped by $f_{-2}^N$ to  an interval containing one of the three intervals $[f_{-2}(q_M), q]$, $\pm[q_M, -q]$. 
         Some definite proportion of such an interval gets mapped into $U_M$ in at most $2$ iterates. Therefore, some definite proportion of $E_N$ gets mapped into $U_M$ in at most $N+2$ iterates. 
         Thus $\Leb(E_N)$ is exponentially small in $N$, as required. 
        \end{proof}

    \section{Cantor repellers}\label{sec:CanRep}
    \begin{defi}\label{defi:rep}
Suppose that $D_1,\dots, D_n$ is a collection of open and non-degenerate topological discs with pairwise disjoint closures each compactly contained in a topological disc $D \subset \C$.
A map $\vp:\bigcup_{i=1}^n D_i\to D$ which is biholomorphic onto $D$ on every $D_i$, $1\leq i\leq n$, is called   a 
{\em Cantor repeller}.
        If we allow $n=\infty$, we call $\vp$ a \emph{generalised Cantor repeller}.
\end{defi}
If $\vp$ preserves the real line
and each branch domain $D_i$ is symmetric with respect to $\R$ then $\vp$ is 
a {\em real} Cantor repeller. 
With respect to a map $f$, if there are integers $n_i$ such that $\vp_{|D_i} = f^{n_i}$, we say that $\vp$ is \emph{induced} (by $f$).

If each branch of the inverse of $\vp$ extends univalently to be defined on a disc $\hat D \supset D$ and has range in $\hat D$ then we say that $\vp$ is \emph{extensible over $\hat D$. }
The distortion of each branch of each iterate of $\vp$ can then be bounded in terms of the modulus of the annulus $\hat D \setminus D$ via the Koebe Distortion Theorem. 

Indeed, let $\Delta_r$ denote the distortion constant 
$$
\Delta_r := \sup_g
\sup_{z,z' \in B(0,r)} \left| \frac{Dg(z)}{Dg(z')}\right|,
$$
taking the supremum
over all univalent maps $g$ defined on the unit disc. If $r<1$, $\Delta_r$ is finite. As $r$ tends to $0$, $\Delta_r$ tends to $1$, see \cite[Theorem~1.6]{CarGam}. 

 

\begin{lem}\label{lemQa}
    Let $\gamma \geq0$.
    Let $0<r< \kappa$ and let $\vp : \calD \to D$ be a generalised real Cantor repeller with domain $\calD$ and range $D \subset B(0,r)$, extensible over $B(0,\kappa)$.
   If $2r- \Leb\left(\calD \cap \R\right) \leq \gamma$ then, for all $j \geq 1$,
    \begin{eqnarray*}
        \left(1-\Delta_{r/\kappa} \frac{\gamma}{2r}\right)^j\Delta^{-1}_{r/\kappa} & \leq&
    \inf_{z\in D} \sum_{y \in \vp^{-j}(z)} |D\vp^j(y)|^{-1} \\
        & \leq&
     \sup_{z\in D} \sum_{y \in \vp^{-j}(z)} |D\vp^j(y)|^{-1} \leq \Delta_{r/\kappa}.
    \end{eqnarray*}
\end{lem}

\begin{lem} \label{lemTailsComp}
    Consider two families of maps $F_\eps : X_\eps \to \C$, $\vp_\eps : Y_\eps \to \C$ with uniform exponential tails, whose domains $X_\eps, Y_\eps$ are subsets of a bounded region $U \subset \C$. 
    Suppose that
             each connected component  of the domain of $\vp_\eps$ is mapped biholomorphically by $\vp_\eps$ with uniformly bounded distortion to the same open topological disc $V_\eps$, $V_\eps \supset X_\eps$; 
    Then
            $F_\eps \circ \vp_\eps$ has uniform exponential tails. 
\end{lem}
In applications, $\vp_\eps$ will be generalised Cantor repellers and $F_\eps$ the first entry to a subdomain. 
\begin{proof}
    We shall sketch the proof supposing that the maps are affine on each connected component of their domains, and that the diameters of all concerned sets are in $\{e^{-n} : n \in \Z\}$, that $\diam(V_\eps)=1$,  and we omit the dependence of maps on $\eps$. Suppose that the uniform tail constants for both families are $C, \kappa$. 
    
    There is a bijection between connected components of $F \circ \vp$ 
    and pairs $(D, D')$,
    where 
    $D, D'$ are connected components of the domains of $\vp, F$ respectively and 
$D \cap \vp^{-1}(D')$ is the corresponding component of the domain of $F\circ \vp$. Moreover,
    $$
    \diam(D \cap \vp^{-1}(D')) = \diam(D)\cdot\diam(D'). 
    $$
    Then 
    $$ 
    \sum \diam(W) \leq \sum_{j=0}^n Ce^{-j\kappa} Ce^{(-n+j)\kappa} = C^2 (n+1)e^{-n\kappa} \leq C' e^{-n\kappa'},$$
    for some $\kappa'>0$,
    where the first sum is over components $W$ of the domain of the domain of $F \circ \vp$ of diameter $e^{-n}$.

    The general case is merely a little more cumbersome to write, using ranges and inequalities. 
\end{proof}

\begin{defi} \label{def:repcon}
Let $0 \in \cE \subset \C$ and let 
    $$\{\vp_\eps\}_{\eps \in \cE} : \bigcup_{j\geq1} D_{j,\eps} \to D_\eps \subset \C$$  be a family of Cantor repellers,  extensible over an open topological disc containing $\overline{D_0}$. 
    We say that $\vp_\eps$ \emph{converges} to $\vp_0$ as $\eps \to 0$, or that $\{\vp_\eps\}_\eps$ is a \emph{convergent family}, if the following statements hold: 
     $D_\eps \to D_0$;
     $$\lim_{j_0\to\infty} \sup_{\eps\in \cE,\, j\geq j_0} |D_{j,\eps}| = 0;$$
      for each $j$,
    $D_{j,\eps} \to D_{j,0}$
    and $\vp_\eps \to \vp_0$ uniformly on $D_{j,\eps} \cap D_{j,0}$ as $\eps \to 0$ in $\cE$. 
\end{defi}

Anticipating \S\ref{secTDF}, truncating to a finite number of branches leads to convergence of pressure functions. 
\begin{lem} \label{lemTrunc}
    Let $\{\vp_\eps\}_\eps$ be a convergent family of Cantor repellers, as above, with range $D_\eps \ni 0$.  If each $\vp_\eps$ has $N$ branch domains, then given $\delta>0$, there exists $\eps_0>0$ such that, 
    for all $|\eps|<\eps_0$,  $n\geq 1$ and $t$ in a fixed neighbourhood of $1$,
    $$
  \left| \frac1n \log \sum_{z \in \vp_\eps^{-n}(0)} |D\vp_\eps^n(z)|^{-t}
  - 
  \frac1n \log \sum_{z \in \vp_0^{-n}(0)} |D\vp_0^n(z)|^{-t} \right| < \delta.$$
\end{lem}
\begin{proof}
    The maps are piecewise expanding, so the preimages of 0 move continuously (at $\eps=0$), uniformly in $n$, and the ratio of derivatives at corresponding points is close to 1 for $\eps$ small. 
\end{proof}

%

\section{Binary tree decomposition} \label{sectBinary}
For letters $a,b$, we can consider the set of words of length $n$, $\Sigma_n := \{a,b\}^n$. Overloading the notation, we can view $a,b$ as representing numbers. Then 
$$(a+b)^n = \sum_{\omega \in \Sigma_n} \pi(\omega),$$
where $\pi(\omega)$ is the product of the letters in $\omega$. Each word $\omega$ represents a path in a binary tree. 

Given a generalised Cantor repeller $\vp: \calD \to D \ni 0$, our goal is to estimate the pressure
\begin{equation}\label{eqn:pressdef}
P(t) = \lim_{n\to \infty} \frac1n \log \sum_{z \in \vp^{-n}(0)} |D\vp^n(z)|^{-t}
\end{equation}
for $\vp$, when $t\approx 1$.
Noting that the Chain Rule renders the derivative multiplicative and multiplication is distributive over addition, we shall decompose the sum following the binary tree. 

Given a partition of the domain $\calD$ into two unions of branch domains, $D_a, D_b$, 
 let $\vp_a$ denote the restriction of $\vp$ to $D_a$ and $\vp_b$ the restriction to $D_b$.

 For $\kappa =a,b$, set  $$\overline{Q}_\kappa(j,t) := 
 \sup_{z_0\in D}
 \sum_{z \in \vp_\kappa^{-j}(z_0)} |D\vp^j(z)|^{-t}.$$
We define $\underline{Q}_\kappa(j,t)$ similarly, with the infimum in place of supremum.

\begin{prop}\label{propBTD}
    Suppose there are numbers $\alpha_i, \beta_i>0$, $i=1,2$, and $C\geq 1$,  with
    $$C^{-1} \alpha_1^j \leq \underline{Q}_a(j,t) \leq \overline{Q}_a(j,t) \leq C \alpha_2^j;$$ 
    $$
    \beta_1 \leq 
     \underline{Q}_b(1,t) \leq \overline{Q}_b(1,t) 
     \leq \beta_2,
     $$
     for all $j\geq 1$. 
     Then 
     $$
     \log(\alpha_1 + C^{-1}\beta_1) \leq P(t) \leq \log(\alpha_2 + C\beta_2).$$
\end{prop}
\begin{proof}
 If $\omega = a^{j_1}b^{j_2}\ldots a^{j_{2k-1}}b^{j_{2k}} \in \Sigma_n$, $j_i \geq 1$ except possibly $j_1, j_{2k}$ which could be $0$, 
    we can write $$S_\omega := \vp_b^{-j_{2k}}\circ \cdots \vp_a^{-j_1}(0).$$ Then there is a \emph{binary tree decomposition} of preimages, 
 $$
    \vp^{-n}(0) = \bigcup_{\omega \in \Sigma_n}S_\omega.$$
For one branch $\omega$ of this tree, and dropping the dependence on $t$ from notation, we estimate
$$
\underline{Q}_a(j_1)\underline{Q}_b(j_2)\ldots \underline{Q}_b(j_{2k})
\leq
\sum_{z\in S_\omega} |D\vp^n(z)|^{-t} 
\leq
\overline{Q}_a(j_1)\overline{Q}_b(j_2)\ldots \overline{Q}_b(j_{2k}).$$
The lefthand side we estimate from below by
$$
    C^{-1} \alpha_1^{j_1 + \cdots + j_{2k-1}} (C^{-1}\beta_1)^{j_2 + \cdots + j_{2k}}.
    $$ 
Summing over $\omega \in \Sigma_n$, we obtain the lower bound
$$
    C^{-1} (\alpha_1 + C^{-1}\beta_1)^n.$$
    This gives the lower bound for $P(t)$. 
    The upper bound is similar.
\end{proof}

\section{Thermodynamic formalism}\label{secTDF}
While one can frequently bound Hausdorff dimension from above directly, conformal measures and \emph{thermodynamic formalism} have proven useful to obtain lower bounds.
For background on the well-developed theory of thermodynamic formalism (over an infinite alphabet/number of branches), see \cite{SARIG_1999, MUbook, PUbook}. 

If a generalised Cantor repeller $\vp : \cD \to D$ has exponential tails, it admits a
 $(P(t),t)$-conformal probability measure  $m_{\vp,t}$ for $t$ in a neighbourhood of $1$. The measure, by definition, has Jacobian $e^{P(t)}|D\vp|^t$, where the pressure $P(t)$ at $t$ was defined in~\eqref{eqn:pressdef}.
 Given a family of Cantor repellers $\vp$ with uniform exponential tails, $P(t)$ is finite (one can check by hand) $m_{\vp,t}$ exists for $t$ in a fixed neighbourhood of $1$ (applying \cite[Theorem~4]{SARIG_1999}). If there is a $t_0$ with $P(t_0)=0$, then it is unique and the Hausdorff dimension of the 
 set
 $$ 
 \Lambda := \bigcap_{n\geq0} \vp^{-n}(D)$$
 is equal to $t_0$ \cite[Theorem~3.15]{MU_1996}. 

The measures $m_{\vp,t}\circ \vp^{-n}$ converge weakly to the equilibrium measure  $\mu_t$  for $\vp$, the unique invariant probability measure equivalent to $m_{\vp,t}$.  
Suppose that $\vp$ has range contained in $B(0,r)$ and that $\vp$ is extensible over the unit disc. 
    By the standard Folklore Theorem construction, 
    \begin{equation}\label{eqnFolk}
        \Delta_r^{-t} \leq \left|\frac{d\mu_{\vp,t}}{dm_{\vp,t}} \right| \leq \Delta_r^{t}.
    \end{equation}
    Knowing the Jacobian of $m_{\vp,t}$, if $W$ is a branch domain of $\vp$ and $z \in W$ then 
    $$
    \Delta_r^{-t}m_{\vp,t}(W) \leq |D\vp(z)|^{-t}e^{-P(t)} \leq \Delta_r^tm_{\vp,t}(W)$$
    and, for all $n\geq 1$,
    \begin{equation}\label{eqn:pressalpha}
    \Delta_r^{-t} \leq \sum_{z \in \vp^{-n}(0)} |D\vp^n(z)|^{-t}e^{-nP(t)} \leq \Delta_r^t.
    \end{equation}

    Let $0 \in \cE \subset \C$ and let 
    $$\{\vp_\eps\}_{\eps \in \cE} : \bigcup_{j\geq1} D_{j,\eps} \to D_\eps \ni 0$$  be a convergent (see Definition~\ref{def:repcon}) family of generalised Cantor repellers with uniform exponential tails,  extensible over some topological disc $\hat D$ which compactly contains $D_0$. 

    We denote by $P_\eps$ the pressure function of $\vp_\eps$. 

    \begin{lem} \label{lemPconverges}
        As $\eps \to 0$, $P_\eps(t) \to P_0(t)$ uniformly for $t$ in a neighbourhood of $1$. 
    \end{lem}
    \begin{proof}
        Given $\rho>0$, let $D_a$ be the union of branch domains of $\vp_0$ of diameter at least $\rho$. There are only finitely many, so we can define $D_a(\eps)$, the union of the  corresponding branch domains  of $\vp_\eps$. 
        Let $D_b = D_b(\eps)$ denote the union of the remaining domains. For $\eps$ small enough, the diameter of any branch domain contained in $D_b$ will have diameter bounded by $2\rho$. Let 
        $$
        \beta(\eps,t) = \sum_{z \in \vp_\eps^{-1}(0) \cap D_b} |D\vp_\eps(z)|^{-t}.$$
        From uniform exponential tails, 
        for all $\delta>0$ there is a $\rho>0$ for which 
        $$\beta(\eps,t) < \delta,$$ for all $t$ in a neighbourhood of $1$. 

        We define $P_{a,\eps}(t)$ to be the pressure at $t$ of $\vp_\eps$ restricted to $D_a$. 
        By Lemma~\ref{lemTrunc}, for small enough $\eps$, 
        $$
        |P_{a,\eps} - P_{a,0}| < \delta.$$
    
        Using~\eqref{eqn:pressalpha}, 
         the hypotheses of Proposition~\ref{propBTD} then hold for $t$ near $1$ with $\alpha_1 = \alpha_2 = e^{P_{a,\eps}(t)}$, and 
         $\beta_1 =\Delta^{-1} \beta(\eps,t)$,  
         $\beta_2 =\Delta \beta(\eps,t)$, where $\Delta$ is a distortion bound. 
         Applying the proposition gives that $|P_\eps(t) - P_{a,\eps}(t)|$ can be made arbitrarily small, if $\beta$ is. 
         The result follows. 
    \end{proof}

    We have shown convergence of pressure in a neighbourhood of $t=1$;  we need an estimate on the slope of $P_\eps(t)$.

    \begin{lem} \label{lemPslope}
        With $\vp_\eps$ as before, suppose that $D_\eps \subset B(0,r)$, $\Delta_r < 11/10$  and $\hat D = B(0,1)$, and suppose $|D\vp_\eps| >2$.

        Given $\delta >0$, there is a neighbourhood of $t=1$ on which 
        $$\left|\log \frac{P_\eps'(t)}{ P_0'(1)}\right| < \delta + 9\log \Delta_{r},$$
        for all small $\eps$. 
    \end{lem}
    \begin{proof}
        Writing $\mu_{\eps,t}$ for $\mu_{\vp_\eps,t}$ and similarly $m_{\eps,t}$ for the conformal measure,
        the Lyapunov exponent $\chi_{t,\eps} :=\int \log |D\vp_\eps| \, d\mu_{\eps,t}$
    satisfies
        $$
        P_\eps'(t) = -\chi_{t,\eps}.$$
        By uniform exponential tails, with $\rho$ small enough and $D_a, D_b$ as before, 
        $$
        \int_{D_b} \log |D\vp_\eps| \, dm_{\eps,t}$$ 
        can be made arbitrarily small; by~\eqref{eqnFolk} the same holds for 
        $$
        \int_{D_b} \log |D\vp_\eps| \, d\mu_{\eps,t}.$$ 
        Consequently, we may restrict our attention to $D_a$. 

        On $D_a$, we estimate 
        $$
        \chi_{a,t,\eps} := \int_{D_a} \log |D\vp_\eps| \, d\mu_{\eps,t}$$
        by 
        $$ 
        G(t,\eps) : =
        \sum_{z\in \vp_\eps^{-1}(0) \cap D_a} |D\vp_\eps(z)|^{-1} \log|D\vp_\eps(z)|e^{-P_0(1)}.$$
        Using the conformal instead of the invariant measure introduces a multiplicative error of $\Delta^t_r$; as $P(t)$ is approximately constant, approximating the (conformal measure) mass of a branch domain by $|D\vp_\eps(z)|^{-t}$ gives a multiplicative error of $\Delta^t_r(1+\delta)$;
lastly $\log|D\vp_\eps|$ varies by at most $\log|\Delta_r|$ on each branch domain, crudely estimated by a multiplicative error of $\Delta_r^2$ using $|D\vp|>2$.

        Combining the estimates, 
        $$
        \left| \log \frac{\chi_{t,\eps}}{\chi_{t,0}} \right| 
        \leq
        \left| \log \frac{\chi_{a, t,\eps}}{\chi_{a,t,0}} \right| +\delta
        \leq \left| \log \frac{G(t,\eps)}{G(t,0)} \right| + \log \Delta_r^{4t+4} + 2\delta.$$

        Noting that $G(t,\eps) \to G(t,0)$ as $\eps \to 0$ and replacing $\delta$ by $\delta/3$ completes the proof. 
    \end{proof}

\section{Lyapunov exponent for $f_{-2}$}\label{sec:Lyap}

For the parameter $-2$,  the map $f_{-2}$ is trigonometrically conjugate to the tent map $T(x) = 1-2|x|$ and  has an invariant probability measure $\nu_0$ on $[-2,2]$ with density
 $$
 \frac{1}{\pi\sqrt{4-x^2}}$$ 
 and  Lyapunov exponent
 $$\int \log|Df_{-2}| \, d\nu_0 = \log 2.$$
 Let $q_M  \in \R^+$ be as defined in \S\ref{secumqm}.
 For the first return map\footnote{
     If $f :X \to X$ has an ergodic, invariant, probability measure $\nu$ and $A \subset X$ with $\nu(A) >0$, let $A_n\subset A$ be the subset of points in $A$ with first return time $n$. Then 
     $$
     \nu = \sum_{n\geq1} \sum_{j=0}^{n-1} f^j_* (\nu_{|A_n}).$$}
     $R_{-2}$ to $B(0,q_M)$, applying the Chain Rule,
 $$
 \int_{B(0,q_M)} \log |DR_{-2}| \, d\nu_0 = \int_{[-2,2]} \log|Df_{-2}| \, d\nu_0 =  \log 2 .$$
 Using the density estimate,
 \begin{equation}\label{eqn:LE2}
     \lim_{M \to \infty} \frac{q_M}{\pi} \int \log |DR_{-2}| \, d\mu_0 = \log 2,
 \end{equation}
     where $\mu_0$ is the normalised restriction of $\nu_0$ to $B(0,q_M)$.

 \section{Misiurewicz maps}\label{sec:mis}

 For $\kappa$-Misiurewicz maps $f_c$, $c\in \R$, we have an optimal result. Define $q_M=q_M(c)$ as in~\S\ref{secumqm}. For $M$ large enough, $q_M < \kappa$. 
 Let $\vp$ be the first return map to $U_M=U_M(c)$. By Lemma~\ref{lemUareNice}, $U_M$ is a nice set so $\vp$ is a generalised Cantor repeller. By Lemma~\ref{lemFmTails}, the first entry maps to $U_M$ have uniform exponential tails. Consequently, so do the maps $\vp$  for $c \in B(-2, \eps_1)$, say. Each branch is extensible over $B(0,\kappa)$, when restricted to the $\kappa$-Misiurewicz setting.

 Let $D_a$ denote the union of those branch domains of $\vp$ which intersect the real axis. Let $D_b$ denote the union of those intersecting the imaginary axis. $D_a$ covers  the interval $(-q_M,q_M)$ up to a set of zero measure. 
 As $p(c)$ moves analytically with speed $-1/3$ at $c=-2$, 
\begin{equation}\label{eqn:pest}
    \sqrt{c+p} = \sqrt{2\eps/3 + O(\eps^2)} = \sqrt{2\eps/3} (1+ O(\eps)).
\end{equation}
 Thus the set $iD_b \cap \R$ has one-dimensional Lebesgue measure 
 $2\sqrt{2\eps/3}$, up to a factor of $1+C_0\eps$ for some universal constant $C_0$.
 With $Q_a, Q_b$ as per Section~\ref{sectBinary}, we estimate $Q_a$ using Lemma~\ref{lemQa} and $Q_b$ using the estimate for $iD_b\cap\R$ and bounded distortion.

 The hypotheses of Proposition~\ref{propBTD} hold for $t=1$ with $\alpha_1 = \alpha_2 = 1$, 
 $$\beta_1 = C^{-1}\frac{1}{2q_M}2\sqrt{2\eps/3} (1+C_0\eps)^{-1},$$ 
 $\beta_2 = C^2\beta_1(1+C_0\eps)^2$
 and $C= \Delta_{\diam(U_M)/\kappa}.$ Recall that $q_M(c) \to q_M(-2)$ as $c+2 = \eps \to 0$. 

 Consequently, 
 \begin{equation*} 
     \log\left(1 + \frac{C^{-2}}{q_M}\sqrt{2\eps/3}(1+C_0\eps)^{-1}  \right)
 \leq P(1) \leq 
    \log\left(1 + \frac{C^2}{q_M}\sqrt{2\eps/3}(1+C_0\eps) \right), 
 \end{equation*}
 and, for $\eps$ small enough,
 \begin{equation*}   \frac{C^{-3}}{q_M}\sqrt{2\eps/3}(1+C_0\eps)^{-1}  
 \leq P(1) \leq 
     \frac{C^2}{q_M}\sqrt{2\eps/3}(1+C_0\eps). 
 \end{equation*}

    On a neighbourhood (depending only on $M$) of $t=1$, 
    the slope of the pressure function, $P'(t)$,
    by Lemma~\ref{lemPslope} and \eqref{eqn:LE2}, 
    equals  $-\frac{\pi}{q_M} \log 2$, up to a factor of $\delta_1(M)>1$, with $\delta_1(M) \to 1$ as $M \to \infty$. Recall that, for each $M$, estimates hold for sufficiently small $\eps$. 
    By Lemma~\ref{lemPconverges}, $P(1) \to 0$ as $\eps \to 0$. 

    Let $\Lambda := \bigcap_{n\geq0} \vp^{-n}(U_M(c))$. 
    Recall that $\HD(\Lambda) = P^{-1}(0).$
    Elementary geometry together with the above estimates for $P(1)$ and $P'(t)$ in a neighbourhood of $1$ give, for small enough $\kappa$-Misiurewicz parameters $\eps>0$,
    $$
    C^{-3}
\delta_1^{-1}(M) 
(1+C_0\eps)^{-1} 
<
    \frac{
        \HD(\Lambda) -1}{\Omega \sqrt{\eps}} 
     < 
     C^2 \delta_1(M) 
 (1+C_0\eps)  ,
    $$
    where $\Omega = \sqrt{\frac{2}3} \frac{1}{\pi \log2}$. As $M$ increases, $q_M \to 0$ and $\diam(U_M) \to 0$ so $C \to 1$. 

    Since $\HD(\Lambda) = \HD(\J_c)$, we deduce that, provided that we restrict to real $\kappa$-Misiurewicz parameters, 
    $$
    \lim_{c \to -2}
    \frac{
        \HD(\J_c) -1}{\Omega \sqrt{\eps}} = 1,
        $$
        completing the proof of Theorem~\ref{theo:upper}.

\section{Extensibility, inducing and exponential tails}\label{sec:tails}
    The topological construction of our (generalised) Cantor repeller culminates in Proposition~\ref{prop:CantorGoal}, following several refining inductive ingredients. 

    Proposition~\ref{prop:linkage} says that the branch domains can be separated into two groups,  one collection $D$ going almost horizontally from $-q_M$ to $q_M$ and another collection $D'$ joining the two points $f_c^{-1}(-p)$ almost vertically, each with controlled gaps. 

\subsection{The integer $N = N(c)$}
The integers $N+1, N+2$ will be the return times of branches of the first return map to $U(c)$ which may have large or uncontrolled distortion (for example, one such (\emph{central}) branch domain may contain $0$, the others will be adjacent). 

Recall that $\zeta$ is the inverse branch of $f_c$ fixing $p(c)$.
\begin{lem} \label{lem:zetanspeed}
    Suppose $c \mapsto w(c)$ is holomorphic and $w(c) \in B(1,2.2)$. Let $w_n(c) := \zeta^n(w(c))$. 
    For $\eps$ small and $n \geq 10$, 
    $$
    |w'_n(c) + 1/3| < 1/10 + |w'(c)| |D\zeta^n(w)|.
    $$
\end{lem}
\begin{proof}
Apply \eqref{eqngenspeed} and the derivative estimate $Df_{-2}(p(-2)) = 4$, noting that $
\sum_{j\geq1} 4^{-j} = \frac13.$
\end{proof}

For $\eps$ small, the sets $\zeta^n(B(0,1.2))$ form a chain of topological discs leading towards $p$. 
Indeed, $\sqrt{2} \approx \zeta(0) \notin B(0,1.2)$ and $0 \notin \zeta(B(0,1.2))$, from which it follows that if  
$0\leq m < n$  then 
$$\zeta^n(B(0,1.2)) \cap \zeta^m(B(0,1.2)) \ne \emptyset$$ if and only if $m+1=n$. 
If $\eps$ is real, then the topological discs are symmetric about the real axis. 

For $\eps = c+2 >0$ real, we let $N(c)$ be the minimal integer  
     with $c \in -\zeta^N(B(0,1.2))$. If $m \notin \{N, N+1\}$ then $c \notin -\zeta^m(B(0,1.2))$. This defines a function $N : (-2, -1) \to \R$. We extend it to complex values by $N(c) = N(\Re(c))$. 

     Henceforth we drop the dependence of $c$ from notation and just write $N$. 

\begin{lem} \label{lem:cleanbranches}
    There exist $\delta >0$ such that, if $c = -2 + x + iy$ with $0< x < \delta$ and $|y| < \delta x$,  if  $n \notin \{N, N+1\}$ then $\dist{c, -\zeta^n(B(0,1.1))} > \delta x$. 
     \end{lem}

\begin{proof} 
    Let $c_0 = -2 + x$.
    For $n \notin \{N, N+1\}$, 
    $$
    \dist{c_0, -\zeta_{c_0}^n(B(0, 1.15))} ~\grtsim~ x.$$
    Hence, by Lemma~\ref{lem:zetanspeed}, for all $c' \in B(c_0, \delta x)$,
    $$\dist{c', -\zeta_{c'}^n(B(0, 1.15))} ~\grtsim~ x.$$
\end{proof}
\begin{lem} \label{lem:vwdist}
    There exist $\delta >0$ such that, if $c = -2 + x + iy$ with $0< x < \delta$ and $|y| < \delta x$,  if  $n \geq N$ then $-\zeta^n(B(0,1.1)) \subset B(c, 20\eps)$. 
     \end{lem}
\begin{proof} 
    For $c_0 = -2+x$, let $V_n = \zeta_{c_0}^n(B(0,1.1)) \cap \R$. 
    Recall that $p(c_0) + c = \frac23 x + O(x^2)$.
    Thus $|V_{N+2}| \subset B(p(c_0), x)$ so $V_n \subset B(p(c_0), 16x)$ for $n\geq N$. 
    By \emph{Poincar\'e discs} \cite[Fact 2.1.2]{Graczyk_1998}, if $D_n$ is the disc with diameter $V_n$, 
    $$
    \zeta^n_{c_0}(B(0,1.1)) \subset D_n \subset B(p(c_0), 16x).$$
    Then use Lemma~\ref{lem:zetanspeed} again. 
\end{proof}

\subsection{Initial inducing} \label{sec:initin}
Recall that $U(c)$ is the fundamental domain defined in~\S\ref{sec:funddom}. The following lemma gives definite Koebe space between $U(c)$ and the boundary of a fixed larger domain $\hat U$, for all $c$ near $-2$. 
\begin{lem} 
    There exists a fixed topological disc $\hat U$ with $\overline U(-2) \subset \hat U \subset B(0,1.1)$  with 
    $$f_c(U) \supset -\zeta_c^n(\hat U)$$
    for all $n>1$ and $c=-2+\eps$, $\eps$ sufficiently small. 
    Moreover, $\hat U$ can be chosen so that $f_c^{-1}(-\zeta_c(\hat U)) \subset \hat U$. 
\end{lem}
\begin{proof}
    The boundary of $U(-2)$ is made of two dynamical rays and two elliptical arcs. Recalling the first part of Lemma~\ref{lemJellipse}, the ellipse gets mapped by $f_{-2}$ outside itself, and preimages are inside. 
    The two dynamical rays through $q, -q$ are mapped by $f_{-2}$ to the dynamical ray through $q$. 
    The set $\zeta_{-2}(U(-2))$ shares a common boundary with $U(-2)$ along the dynamical ray through $-q$. The sets $\zeta_{-2}^n(U(-2))$, $n>1$ are at a definite distance from the two dynamical rays and from the boundary of the ellipse. This remains true replacing $U(-2)$ by a neighbourhood $\hat U$ compactly containing $U$ and $-2$ by $-2+\eps$, for $\eps$ small. 

    For the case $n=1$, use that, in a neighbourhood of the two dynamical rays, the derivative of $f_{-2}$ is bigger than $1$. 
\end{proof}

    We denote by $\phi_0$ the first return map $\phi_0$ for $f_c$ to $U(c)$.  
    Let $V_1^+$ denote the connected component of $f_c^{-1}(-\zeta_c(U(c)))$ containing $q$ in its boundary, and $V_1^-$ the other, symmetric, connected component. These domains are the two boundary-adjacent domains.
            Denote by $V_N^\pm, V_{N+1}^\pm$, the (three or four) connected components of $f^{-1}_c(-\zeta_c^k(U(c)))$, $k \in \{N, N+1\}$. 
\begin{coro} \label{cor:extens}
    In $U(c)$, the branch domains of  $\phi_0$ are all mapped
    \begin{itemize}
        \item
            univalently onto $U(c)$ by $\phi_0$ and are
        \item
            extensible over $\hat U$ with extension domain contained in $U$,
    \end{itemize}
    except possibly 
    \begin{itemize}
        \item
            $V_1^{\pm}$, whose extension domains are not contained in $U$ but are in $\hat U$;
        \item
            $V_N^\pm, V_{N+1}^\pm$, whose extension domains may contain $0$. 
    \end{itemize}
    For $k \in \N\setminus \{1, N+1, N+2\}$, there are precisely two branch domains on which $\phi_0 = f_c^k$ (so the \emph{return time} is $k$). 

    Two extension domains (with return time not in $\{N+1, N+2\}$) overlap only if the corresponding branches are adjacent. 
\end{coro}

Denote by $\phi_1$ the restriction of $\phi_0$ to domains with return time bounded by $N$. 
Note that, by Lemma~\ref{lem:expandingreturns}, $|D\phi_1| > 2$ and $|Df_c^2| > 2$ on the extension domains of $V_1^\pm$. 

On the boundary branch domain $V_1^+$ say, containing $q$ in its boundary, we do some inducing to obtain extension domains inside $U = U(c)$. Let $\psi_1$ denote the first entry map from $V_1^+$ to $U\setminus V_1^+$. 
We define $\phi_2$ on $V_1^+$ to be $\phi_1 \circ \psi_1$. Define it on $V_1^-$ symmetrically, $\phi_2(z) = \phi_2(-z)$, and let $\phi_2 = \phi_1$ elsewhere (that is, on the branch domains of $\phi_1$ with return time between $3$ and $N$ inclusive). 
\begin{lem} \label{lem:phi2}
    The branch domains of $\phi_2$ are mapped by $\phi_2$ onto $U(c)$, extensibly over $\hat U$, with extension domains contained in $U(c)$. 
\end{lem}
\begin{proof}
    This follows from the corollary, noting that the pullback of the extension domain of $V_1^-$ by the branch of $f_c^{-2}$ which fixes $q$ lies in $U$, so only the boundary-adjacent branch of $\phi_0\circ f_c^2$ restricted to $V_1^+$ has an extension domain not in $U$. Then continue inducing on the new boundary-adjacent branch.
\end{proof}

Since $f_c^{2}$ is expanding on $V_1^+$, we control how many branch domains we create at each scale. 
\begin{lem}
    There is a constant $C>1$ such that, for all $n$, there are at most $C$ branch domains of $\phi_2$ with diameter in the range $[e^{-n-1}, e^{-n}]$. 
\end{lem}

\subsection{Postcritical filling}\label{sec:pcf}
Let $D_k$ denote the collection of branch domains of $\phi_k$, for $k=1,2$.  
    If $W$ is mapped biholomorphically by the $n_W$th iterate of $f_{c}$ to $U = U(c)$, we can define $W(D_k)$ to be the collection of branch domains of $\phi_k \circ f^{n_W}$ contained in $W$.

    We proceed with an iterative construction of $Z_{j+1}$, starting from $j=1$, where we set $Z_1 = D_1$, $Z_0 = \emptyset$.
    There is at most one element (branch domain)  $W \in Z_j$ which contains $\phi_0(0)$. 
     If it exists, it is mapped biholomorphically by $\phi_1^j$ to $U$.
    In $Z_j$, replace $W$ with the elements of $W(D_1)$. 
    There are at most two domains in $Z_j$ not containing $\phi_0(0)$ but   whose extension domains contain $\phi_0(0)$.   For each such domain $W'$, 
    in $Z_j$, replace $W'$ with the elements of $W'(D_2)$. 

    By construction of $\phi_2$, the extension domains of elements of $W'(D_2)$ are subsets of $W'$ and hence do not contain $\phi_0(0)$. There are at most two elements of $W(D_1)$ whose extension domains contain $\phi_0(0)$. 
Note that the diameters of these two elements are $\lesssim~2^{-j}$ as $|D\phi_1| > 2$.

    In the limit, we obtain $Z_\infty$, a collection of domains each mapped by an iterate of $\phi_1$ biholomorphically to $U(c)$ and extensible over $\hat U$, with extension domain not containing $\phi_0(0)$. 
    Denote the corresponding map by $\phi_\infty$. 

    Let 
    \begin{equation}\label{eqn:Vstar}
        V_* = V_N^+\cup V_N^- 
        \cup
         V_{N+1}^+\cup V_{N+1}^-
    \end{equation}
and
    define $\phi_*$ on $V_*$ by 
    $$
    \phi_* = \phi_\infty \circ \phi_0.$$

    \begin{lem} \label{lem:startails}
        The map $\phi_*$ has the following properties:
        \begin{itemize}
            \item
                Each branch domain is mapped onto $U(c)$ and is extensible over $\hat U$;
            \item
                $\phi_*$ has uniform exponential tails. 
        \end{itemize}
    \end{lem}
    \begin{proof}
        Extensibility is by construction. 

        Iterates of $\phi_1$ have uniformly bounded distortion, by extensibility. At the $j$th step in the inductive construction of $Z_\infty$, the elements whose extension domains contained $\phi_0(0)$ had diameter $\lesssim~2^{-j}$. Consequently, there are no new elements of scale in the range $[e^{-n-1}, e^{-n}]$ after $\sim n$ induction steps. 
        Hence there are at most $\sim~ n$ domains in that range. Therefore, as triangular numbers grow quadratically, there are at most $\sim n^2$ elements of diameter greater than $e^{-n}$.  

        On a connected component $W$ of  $V_*$, $\phi_0 = f_c^k$ for $k = N+1$ or $N+2$. 
        On $f_c(W)$, $f^{k-1}_c$ has bounded distortion and derivative in absolute value comparable with $|1/\eps|$. 
        A branch domain of $\phi_*$ of diameter close to $e^{-m}$ is mapped by $\phi_0$ to a domain of $\phi_\infty$ of diameter at least $\sim ~ e^{-2m}/|\eps| > e^{-2m}$. There are at most $\sim ~ m^2$ of those. 
        
        As $m^2e^{-m} < e^{-m/2}$ for large $m$, this shows exponential tails. 
    \end{proof}

    \subsection{The final induction step} \label{sec:definevp}
    Recall $F_M$ is the first entry map to $U_M = U_M(c)$, with $U_M$, $q_M$ defined in \S\ref{secumqm}. 
    We define $\vp = \vp_\eps$ by
    $$
    \vp = F_M \circ \phi_* \text{ on $V_*$}
    $$
    and 
    $$
    \vp = F_M \circ \phi_0 \text{ elsewhere}.
    $$

    \begin{prop} \label{prop:CantorGoal}
The family of maps $\vp$ has the following properties:
         \begin{itemize}
            \item
                Each branch domain is mapped onto $U_M(c)$ and is extensible over $U(c)$;
            \item
                the family of maps $\vp$ has uniform exponential tails. 
        \end{itemize}
    \end{prop}
    \begin{proof}
        Domains of $F_M$ are mapped univalently onto $U_M$ and are extensible over $U$ 
        with extension domains contained in $U$. 
        Domains of $\phi_*$ and domains of $\phi_0$ disjoint from the domain of $\phi_*$ are mapped univalently onto $U$, extensibly over $\hat U$ and hence with uniformly bounded distortion. 
        The topological statement follows. 

        As these (families of) maps all have uniform tails (Lemma~\ref{lem:startails} and  Lemma~\ref{lemFmTails}), so does their composition, by Lemma~\ref{lemTailsComp}. 
    \end{proof}

    \subsection{Some geometric estimates}

    If $W$ is a branch domain of $F_M\circ \phi_*$ mapped by $f_c^{n_W}$ to $U_M$, there are  unique points $v_W^\pm \in \partial W$ mapped by $f_c^{n_W}$ to $\pm q_M(c) \in \partial U_M$. 

    \begin{lem}\label{lem:branchboundary}
        Let $\theta \in (0, 1/12)$. 
        If $c+2 = \eps = x+iy$ with $0 \leq |y| < x^{1+5\theta}$,
        $$\dist{f_c(v_W^+), \R}+ \dist{f_c(v_W^-), \R} \leq  x^{1+4\theta}.$$ 
    \end{lem}
    \begin{proof}
        Let $v_W$ denote one of $v_W^\pm$. 
        By construction, $f_c^j(W) \cap V_* = \emptyset$ for $j=1, \ldots, n_W-1$ so the orbit of $f_c(v_W)$ avoids, by Lemma~\ref{lem:cleanbranches},  $B(0, \sqrt{\delta x})$ for  some independent $\delta>0$.
        
        The orbit of $f_c(v_W)$ is a finite repeller, contained in the set $K$ of 
         Corollary~\ref{cor:holospeed}, with a corresponding holomorphic motion $h$ satisfying 
         $$
         |\partial_1 h| \leq |\eps|^{-1/2 -2\theta}$$
            on $B(c, |\eps|^{1+4\theta})$. 
         This give us an estimate for $\rho'$ where $\rho(c) = f_c^{N+1}(v_W)$. 
         We have that $|Df_c^N(f_c(v_W))| ~\sim~ x^{-1}$. 
         By Lemma~\ref{lem:thirdquarter} and~\eqref{eqngenspeed}, we obtain 
         $$
         |\partial_1 h(c', f_c(v_W))| \leq 1/2.$$
          Therefore, for $c' \in B(c, |\eps|^{1+4\theta})$, 
         $$
         \dist{h(c',f_c(v_W)), f_c(v_W)} \leq |\eps|^{1+4\theta}/2.$$

         It remains to check that if $c' = -2+x$, $h(c',f_c(v_W)) \in \R$. This follows from the construction of $\phi_\infty$, which uses iterates of $\phi_1$. 
    \end{proof}
When $\eps$ is real, the branch domains of $\vp$ lie along the real and imaginary axes. We need the following lemma, for the complex case, to show that there are branch domains close to $0$. 
    \begin{coro}\label{cor:bothsides}
        If $f_c(v_W^\pm) -c$ are on alternate sides of the imaginary axis, then 
        $$v_W^\pm \in B(0, 2 x^{1/2+2\theta}).$$ 
    \end{coro}
    \begin{proof}
        Let $s = |\Re(f_c(v_W)-c)|$. 
        We have $\diam(f_c(W)) \geq s$. Applying Lemma~\ref{lem:UMext}, 
        $$
        B(f_c(W), 2s) \subset f_c(\hat W)$$
        and so does not contain $c$, so $\dist{f_c(v_W), c} \geq 2s$. 
        Hence the imaginary part of $(f_c(v_W) - c)$ is at least $\sqrt{3}$ times the real part. 
        As 
        $$|\Im(f_c(v_W)) - \Im(c)| \leq x^{1+4\theta} + x^{1+5\theta}< 2x^{1+4\theta},$$ 
        we obtain that  $\dist{f_c(v_W), c}< 4x^{1+4\theta}.$
    \end{proof}

    \subsection{Linkage}
    As the boundary of $U_M(c)$ spirals in to $q_M$ when $c \notin \R$, we cannot just consider diameters of sets; the notion of \emph{linkage} becomes important. 
    In the inductive construction of $\vp$ we discarded some branches (using the restriction $\phi_1$), creating \emph{gaps}. 

    We say that a set $X$ of subsets of $\C$ $\kappa$-links points 
 $w, w'$ if there is a sequence $S_1,\ldots, S_k$ of elements of $X$ with
$$
\dist{w,S_1}+ \sum_{i=1}^{k-1} \dist{S_i, S_{i+1}})  + \dist{S_k, w'} < \kappa .$$
If it $\kappa$-links for all $\kappa>0$, then we say it $0$-links. 

If $G$ is a countable set of subsets of $\C$ and $X \cup G$ $0$-links $w,w'$ then we say that $X$ links $w,w'$ with \emph{gaps} $G$. 
We say the gaps $G$ sum to 
 $$\sum_{S\in G} \diam(S).$$
 If the gaps sum to less than $\kappa$ then $X$ $\kappa$-links $w,w'$.

 If $W$ is mapped by $f_c^{n_W}$ biholomorphically to a set $V$, we call $W$ a \emph{biholomorphic pullback} of $V$. If $V = U(c)$, there are two marked points in $\partial W$ mapped by $f^n_c$ to $\pm q$. If $V = U_M(c)$, there are two marked points mapped to $\pm q_M$. In either case, denote the set of these two marked points by $b(W)$. 

 If $Z$ is a set of biholomorphic pullbacks of $U(c)$ or of $U_M$, we define
 $$
 \hat b(Z) = \{b(W) : W \in Z\}. $$

 As a simple example worth considering, if $X$ is the set of branch domains of the first entry map $F_M$ to $U_M$, then $\hat b(X)$ 0-links $\pm q$.

 \begin{prop}\label{prop:linkage}
        Let $\theta \in (0, 1/12)$. 
     There is a constant $C>1$ such that,
        if $c+2 = \eps = x+iy$ with $0 \leq |y| < x^{1+5\theta}$,
     the branch domains of $\vp$ split into two sets $D, D'$ of branch domains where
     \begin{itemize}
         \item
             $\hat b(D)$ $C |\eps|^{1/2+2\theta}$-links $\pm q_M$;
         \item
             $\hat b(D')$ $C |\eps|^{1/2+2\theta}$-links the two points $f_c^{-1}(-p)$.
     \end{itemize}
 \end{prop}
 
 \begin{proof}
Recall that $D_1$ is the set of branch domains of $\phi_1$, the first return map to $U(c)$ restricted to domains with return time bounded by $N$. 
Let 
$$\pm v_* = f_c^{-1}(-\zeta^N(q)) \text{ and } \pm w_* = 
f_c^{-1}(-\zeta^{N+2}(q)).$$
     By pulling back the estimates of Lemma~\ref{lem:vwdist},  $|v_*|, |w_*| \leq 10 |\eps|^{1/2}$. 
Then $\hat b(D_1)$ links $\pm q$ with a gap set $G = \{ \pm v_*\}$. 

As an aside, if $D_0$ is the collection of branch domains of $\vp$ outside $V_*$, then
     $\hat b(D_0) \cup \{\{\pm v_*\}\}$ $0$-links $\pm q_M$, while $\hat b(D_0) \cup \{\{\pm w_*\}\}$  $0$-links the two points $f^{-1}_c(-p)$. We need to study the domains inside $V_*$.

The set $D_2$
of branch domains of $\phi_2$, 
defined in~\S\ref{sec:initin},
is a little more complicated. From the construction, discussed in the proof of Lemma~\ref{lem:phi2}, we control how many gaps we create at each scale. The set $\hat b(D_2)$ links $\pm q$ with a gap set $G$ for which 
$$
\sup_{S \in G} \diam(S) < 10|\eps|^{1/2}$$
     and  
$$\# \{ S \in G : \diam(S) \in [e^{-n-1}, e^{-n}] \} \leq 4.$$

Now we estimate the gaps created in~\S\ref{sec:pcf}.
One action was, in $Z_j$, to replace $W$ with $W(D_1)$. By bounded distortion, $\hat b(W(D_1))$ links $b(W)$ with a gap set $G$ consisting of one element of diameter bounded by $C \diam(W) |\eps|^{1/2}.$ 
Another was to replace $W'$ with $W'(D_2)$. Similarly, 
$\hat b(W'(D_2))$ links $b(W')$ with a gap set $G$ for which
$$
\sup_{S \in G} \diam(S) < C\diam(W')|\eps|^{1/2}$$
and
$$\# \{ S \in G : \diam(S) \in [e^{-n-1}, e^{-n}] \} \leq C.$$

We obtained the map $\phi_\infty$ with branch domains $Z_\infty$; we estimate $\hat b(Z_\infty)$ links $\pm q$ with a gap set $G$ for which
$$
\sup_{S \in G} \diam(S) < 10|\eps|^{1/2}$$
and 
$$\# \{ S \in G : \diam(S) \in [e^{-n-1}, e^{-n}] \} \leq Cn.$$

If $X_1$ is the set of branch domains of $F_M \circ \phi_\infty$, then $\hat b(X_1)$ links $\pm q$ with the same gap set $G$, since $\hat b(X)$ $0$-linked $\pm q$, our first example. 

Let $X_2 = \{f_c(b(W)) : W \text{ is a branch domain of } F_M\circ \phi_*\}.$
Then $X_2$ links $f_c(v_*), f_c(w_*)$ with a gap set $G$ for which
$$
\sup_{S \in G} \diam(S) ~\lesssim~ |\eps|^{3/2}$$
and
$$\# \{ S \in G : \diam(S) \in [e^{-n-1}, e^{-n}] \} \leq Cn.$$
We claim that there is an element  $S \in X_2$ with $\dist{S,c} < C x^{1+4\theta}$. Indeed, if $f_c(b(W)) -c$ is on alternate sides of the imaginary axis for some branch domain $W$ of $F_M \circ \phi_*$, then apply Corollary~\ref{cor:bothsides}. 
Otherwise, let $W,W'$ be domains with $b(f_c(W))$, $b(f_c(W'))$ on opposite sides of the vertical line passing through $c$, which minimise 
$$\dist{b(f_c(W)), b(f_c(W'))} ~\lesssim~ |\eps|^{3/2}.$$
   By Lemma~\ref{lem:branchboundary}, the imaginary parts of $b(f_c(W)), b(f_c(W'))$ are small. 
   Hence 
   $$
   \dist{b(f_c(W)), c)} ~\lesssim~ |\eps|^{1+4\theta}.$$
   This shows the claim.  

   In particular, there is a point  $z_* \in b(W)$, $W$ a domain of $\vp$, with $|z_*| ~\lesssim~|\eps|^{1/2+2\theta}.$

   We divide $X_2$ and $G$ into two sets each, $X_3$ linking $f_c(v_*), f_c(z_*)$ and $X_4$ linking $f_c(z_*), f_c(w_*)$ with corresponding gap sets. 

   Let $Y_k = \{b(W) : b(f_c(W)) \in X_k\}$, $k = 3,4$. Then $Y_3$ links $\pm v_*$ with a gap set $G'$ we estimate as follows, one element being $\{\pm z_*\} \in G'$. 
$$
\sup_{S \in G'} \diam(S) ~\lesssim |\eps|^{3/4}$$
and (similarly to the proof of Lemma~\ref{lem:startails})
$$\# \{ S \in G' : \diam(S) \in [e^{-n-1}, e^{-n}] \} \leq Cn^2.$$
Hence
the gaps $G'$ sum to 
$$~\lesssim~|\eps|^{2/3} + 2|z_*|  ~\lesssim~ |\eps|^{1/2+2\theta},$$ so 
$Y_3$ $C|\eps|^{1/2 +2\theta}$-links $\pm v_*$. 

Similarly, $Y_4$ 
 $C|\eps|^{1/2 +2\theta}$-links $\pm w_*$. 
 \end{proof}
 
 Note that if $f^{n_W}_c$ maps $W$ to $U_M$ with distortion bounded by $\Delta>1$, then $\hat b(W(D))$ 
 $$\Delta C |\eps|^{1/2+2\theta}\frac{|r-s|}{2|q_M|}\text{-links}$$ the two points $r,s \in b(W)$.
 Therefore, 
 \begin{equation}\label{eqn:DaW}
 \sum_{S \in \hat b(W(D))} \diam(S) \geq |r-s|\left(1 - 
 \Delta C |\eps|^{1/2+2\theta}\frac{1}{2|q_M|}
 \right). 
 \end{equation}
 Meanwhile
 \begin{equation}\label{eqn:DbW}
 \sum_{S \in \hat b(W(D'))} \diam(S) \geq \Delta^{-1}\frac{|r-s|}{2|q_M|} \sum_{S \in \hat b(D')} \diam(S). 
 \end{equation}

\section{Pressure and dimension for parameters in the cusp}\label{sec:pres}

Given $M\geq 1$, let $\vp = \vp_\eps$ be the map defined (for small $\eps$) in \S\ref{sec:definevp}, $c+2 = \eps = x+iy$, with $x,y \in \R$. 
As $U_M = U_M(c)$ is compactly contained in $U(c)$ and $\vp$, with range $U_M$, is extensible over $U(c)$, there is a uniform distortion bound $\Delta = \Delta(M)>1$ for iterates of $\vp$. 
When $M$ is large, $U_M$ is small, and we can take $\Delta$ close to $1$. 

Each map $\vp$ has a corresponding pressure function $P(\cdot) = P_\eps(\cdot)$. 
In this section we complete the proof of Theorem~\ref{theo:optimal}. Our next step is to prove the following. 
\begin{prop} \label{prop:pressest}
    Let $\hat \theta>0$.
    There is a function $\sigma : \N \times \R^+ \to \R$ with 
    $$
    \lim_{M\to \infty}\lim_{x\to 0} \sigma(M,x) = 0$$
    for which the pressure function satisfies
    $$
    P_\eps(1) \geq \sqrt{\frac23}\frac{1}{|q_M(c)|} \sqrt{|\eps|} (1 + \sigma(M, |\eps|)),$$
    provided $|y| \leq x^{1+\hat \theta}$. 
\end{prop}

\begin{proof}
    Recall that $p(c) + c = \frac23\eps + O(\eps^2)$, so 
    $$|f^{-1}_c(-p)| = \sqrt{\frac23} |\eps|^{1/2} + O(\eps).$$
    Split the branch domains of $\vp$ into $D, D'$ as per Proposition~\ref{prop:linkage}.
    $$
  \sum_{S \in \hat b(D')} \diam(S) \geq 
    2 \sqrt{\frac23} |\eps|^{1/2} - C|\eps|^{1/2+2\theta}.$$
   The constant $C$ depends on $M$, but this term has higher order in $|\eps|$ and will become negligible. 
    Applying~\eqref{eqn:DaW} and~\eqref{eqn:DbW} inductively, and replacing $C$ by a larger constant if necessary,
    if $D_n$ is the collection of branch domains of $\vp^n$, 
    $$
    \sum_{S\in \hat b(D_n)} \diam(S) \geq 
    2|q_M| 
 \left(1 - 
  C |\eps|^{1/2+2\theta} + 
 \Delta^{-1}\frac{1}{2|q_M|} 
    2 \sqrt{\frac23} |\eps|^{1/2} 
    \right)^n,$$
    where $\Delta$ is the distortion bound.
    Hence
    $$
    \lim_{n\to\infty} 
    \frac1n \log \sum_{W\in D_n} \diam(W) \geq \Delta^{-2}\frac{1}{|q_M|}\sqrt{\frac23}|\eps|^{1/2}$$
    for small $|\eps|$.
\end{proof}

The map $\vp$ coincides with the first return map to $U_M$ outside of $V_* \subset B(0, 10\sqrt{|\eps|})$
, so the generalised Cantor repeller converges to the first return map for $f_{-2}$ to $U_M(-2)$ as $\eps \to 0$. 

Choose $0 < R_M < R < 1$  so that $B(0,R) \subset U$ and $B(0, R_M) \supset U_M$, $R_M \to 0$ as $M \to \infty$. 
Apply Lemma~\ref{lemPslope} with $r = R_M/R$ to obtain a neighbourhood of $t=1$ on which
        $$\left|\log \frac{P_\eps'(t)}{ P_0'(1)}\right| < 10\log \Delta_{r},$$
        for all small $\eps$. 
        By Lemma~\ref{lemPconverges}, $P_\eps(1)$ converges to $P_0(1)=0$ (this, together with the slope estimate,  guarantees that $P_\eps^{-1}(0)$ lies within our neighbourhood of $t=1$, for all small $\eps$). 

        Hence, 
        $$
        P_\eps^{-1}(0) - 1 \geq  - \frac{P_\eps(1)}{P_0'(1)}(1 -11\log\Delta_r).$$
        We estimate $P_0'(1)$ by~\eqref{eqn:LE2} and $P_\eps(1)$ by
         Proposition~\ref{prop:pressest}, obtaining
        \begin{eqnarray*}
        P_\eps^{-1}(0) - 1 
            & \geq &
     \sqrt{\frac23}\frac{1}{|q_M(c)|} \sqrt{|\eps|} (1 + \sigma(M, |\eps|))\frac{q_M(-2)}{\pi \log 2}(1-12\log\Delta_r) \\
            &
            = & \Omega \sqrt{|\eps|} \frac{q_M(-2)}{|q_M(c)|} (1 + \sigma(M, |\eps|))(1-12\log\Delta_r) 
        \end{eqnarray*}
     for all small $\eps$, with $c$ within the cusp,  
    where $\Omega = \sqrt{\frac{2}3} \frac{1}{\pi \log2}$. We have that 
    $$
    \lim_{M\to \infty}\lim_{\eps\to 0}
             \frac{q_M(-2)}{|q_M(c)|} (1 + \sigma(M, |\eps|))(1-12\log\Delta_r) = 1.$$

    Let $\Lambda := \bigcap_{n\geq0} \vp^{-n}(U_M(c))$. 
    Recall that $\HD(\Lambda) = P_\eps^{-1}(0)$ and $\Lambda$ is a subset of the Julia set $\J_c$. 

    We obtain, for $\theta>0$ and parameters $c = -2+x+iy$, $0 \leq |y| < x^{1+\theta}$, 
    $$
    \liminf_{c \to -2}
    \frac{
        \HD(\J_c) -1}{\Omega \sqrt{\eps}} \geq 1.
        $$
        Together with the upper bound Theorem~\ref{theo:upper}, this completes the proof of Theorem~\ref{theo:optimal}.

\bibliography{references}
\bibliographystyle{plain}

\end{document}

%% file: largescale.pdf_tex
\begingroup%
  \makeatletter%
  \providecommand\color[2][]{%
    \errmessage{(Inkscape) Color is used for the text in Inkscape, but the package 'color.sty' is not loaded}%
    \renewcommand\color[2][]{}%
  }%
  \providecommand\transparent[1]{%
    \errmessage{(Inkscape) Transparency is used (non-zero) for the text in Inkscape, but the package 'transparent.sty' is not loaded}%
    \renewcommand\transparent[1]{}%
  }%
  \providecommand\rotatebox[2]{#2}%
  \newcommand*\fsize{\dimexpr\f@size pt\relax}%
  \newcommand*\lineheight[1]{\fontsize{\fsize}{#1\fsize}\selectfont}%
  \ifx\svgwidth\undefined%
    \setlength{\unitlength}{564.84921433bp}%
    \ifx\svgscale\undefined%
      \relax%
    \else%
      \setlength{\unitlength}{\unitlength * \real{\svgscale}}%
    \fi%
  \else%
    \setlength{\unitlength}{\svgwidth}%
  \fi%
  \global\let\svgwidth\undefined%
  \global\let\svgscale\undefined%
  \makeatother%
  \begin{picture}(1,0.33682429)%
    \lineheight{1}%
    \setlength\tabcolsep{0pt}%
    \put(0,0){\includegraphics[width=\unitlength,page=1]{largescale.pdf}}%
    \put(0.00059072,0.11159591){\color[rgb]{0.10196078,0.10196078,0.10196078}\makebox(0,0)[lt]{\lineheight{1.25}\smash{\begin{tabular}[t]{l}$-2$\end{tabular}}}}%
    \put(0.06033249,0.1102913){\color[rgb]{0.10196078,0.10196078,0.10196078}\makebox(0,0)[lt]{\lineheight{1.25}\smash{\begin{tabular}[t]{l}$-p$\end{tabular}}}}%
    \put(0.13210956,0.11159629){\color[rgb]{0.10196078,0.10196078,0.10196078}\makebox(0,0)[lt]{\lineheight{1.25}\smash{\begin{tabular}[t]{l}$c$\end{tabular}}}}%
    \put(0.26667279,0.31909687){\color[rgb]{0.10196078,0.10196078,0.10196078}\makebox(0,0)[lt]{\lineheight{1.25}\smash{\begin{tabular}[t]{l}$U$\end{tabular}}}}%
    \put(0.39181207,0.21817444){\color[rgb]{0.10196078,0.10196078,0.10196078}\makebox(0,0)[lt]{\lineheight{1.25}\smash{\begin{tabular}[t]{l}$U_M$\end{tabular}}}}%
    \put(0.46445904,0.12276145){\color[rgb]{0.10196078,0.10196078,0.10196078}\makebox(0,0)[lt]{\lineheight{1.25}\smash{\begin{tabular}[t]{l}$0$\end{tabular}}}}%
    \put(0.50912008,0.12769187){\color[rgb]{0.10196078,0.10196078,0.10196078}\makebox(0,0)[lt]{\lineheight{1.25}\smash{\begin{tabular}[t]{l}$q_M$\end{tabular}}}}%
    \put(0.64136401,0.12493679){\color[rgb]{0.10196078,0.10196078,0.10196078}\makebox(0,0)[lt]{\lineheight{1.25}\smash{\begin{tabular}[t]{l}$-q$\end{tabular}}}}%
    \put(0.75983278,0.24514518){\color[rgb]{0.10196078,0.10196078,0.10196078}\makebox(0,0)[lt]{\lineheight{1.25}\smash{\begin{tabular}[t]{l}$\zeta(U)$\end{tabular}}}}%
    \put(0.87859115,0.18163336){\color[rgb]{0.10196078,0.10196078,0.10196078}\makebox(0,0)[lt]{\lineheight{1.25}\smash{\begin{tabular}[t]{l}$\zeta^n(U)$\end{tabular}}}}%
    \put(0.92774763,0.11986156){\color[rgb]{0.10196078,0.10196078,0.10196078}\makebox(0,0)[lt]{\lineheight{1.25}\smash{\begin{tabular}[t]{l}$p$\end{tabular}}}}%
    \put(0.26477381,0.12493679){\color[rgb]{0.10196078,0.10196078,0.10196078}\makebox(0,0)[lt]{\lineheight{1.25}\smash{\begin{tabular}[t]{l}$q$\end{tabular}}}}%
  \end{picture}%
\endgroup%

%% file: UMphi.pdf_tex
\begingroup%
  \makeatletter%
  \providecommand\color[2][]{%
    \errmessage{(Inkscape) Color is used for the text in Inkscape, but the package 'color.sty' is not loaded}%
    \renewcommand\color[2][]{}%
  }%
  \providecommand\transparent[1]{%
    \errmessage{(Inkscape) Transparency is used (non-zero) for the text in Inkscape, but the package 'transparent.sty' is not loaded}%
    \renewcommand\transparent[1]{}%
  }%
  \providecommand\rotatebox[2]{#2}%
  \newcommand*\fsize{\dimexpr\f@size pt\relax}%
  \newcommand*\lineheight[1]{\fontsize{\fsize}{#1\fsize}\selectfont}%
  \ifx\svgwidth\undefined%
    \setlength{\unitlength}{503.11037559bp}%
    \ifx\svgscale\undefined%
      \relax%
    \else%
      \setlength{\unitlength}{\unitlength * \real{\svgscale}}%
    \fi%
  \else%
    \setlength{\unitlength}{\svgwidth}%
  \fi%
  \global\let\svgwidth\undefined%
  \global\let\svgscale\undefined%
  \makeatother%
  \begin{picture}(1,0.55088485)%
    \lineheight{1}%
    \setlength\tabcolsep{0pt}%
    \put(0,0){\includegraphics[width=\unitlength,page=1]{UMphi.pdf}}%
    \put(0.1575816,0.50898958){\color[rgb]{0.10196078,0.10196078,0.10196078}\makebox(0,0)[lt]{\lineheight{1.25}\smash{\begin{tabular}[t]{l}$U_M$\end{tabular}}}}%
    \put(0,0){\includegraphics[width=\unitlength,page=2]{UMphi.pdf}}%
    \put(0.52936081,0.43945912){\color[rgb]{0.10196078,0.10196078,0.10196078}\makebox(0,0)[lt]{\lineheight{1.25}\smash{\begin{tabular}[t]{l}$f_c^{-1}(-p)$\\\end{tabular}}}}%
    \put(0.49459058,0.08040325){\color[rgb]{0.10196078,0.10196078,0.10196078}\makebox(0,0)[lt]{\lineheight{1.25}\smash{\begin{tabular}[t]{l}$f_c^{-1}(-p)$\\\end{tabular}}}}%
    \put(0,0){\includegraphics[width=\unitlength,page=3]{UMphi.pdf}}%
    \put(0.47472187,0.23225694){\color[rgb]{0.10196078,0.10196078,0.10196078}\makebox(0,0)[lt]{\lineheight{1.25}\smash{\begin{tabular}[t]{l}0\end{tabular}}}}%
    \put(0,0){\includegraphics[width=\unitlength,page=4]{UMphi.pdf}}%
    \put(0.82090907,0.53098199){\color[rgb]{0.10196078,0.10196078,0.10196078}\makebox(0,0)[lt]{\lineheight{1.25}\smash{\begin{tabular}[t]{l}$\vp$\end{tabular}}}}%
    \put(0.00066319,0.2445735){\color[rgb]{0.10196078,0.10196078,0.10196078}\makebox(0,0)[lt]{\lineheight{1.25}\smash{\begin{tabular}[t]{l}$-q_M$\end{tabular}}}}%
    \put(0.92657079,0.25746036){\color[rgb]{0.10196078,0.10196078,0.10196078}\makebox(0,0)[lt]{\lineheight{1.25}\smash{\begin{tabular}[t]{l}$q_M$\end{tabular}}}}%
  \end{picture}%
\endgroup%